\documentclass[aop,preprint]{imsart}
\arxiv{1111.2080}
\usepackage{amssymb}
\usepackage{graphicx}
\usepackage{amsmath, amsthm}
\usepackage{epsfig}

\startlocaldefs
\newtheorem{theorem}{Theorem}

\newtheorem{corollary}[theorem]{Corollary}

\newtheorem{definition}[theorem]{Definition}

\newtheorem{lemma}[theorem]{Lemma}

\newtheorem{problem}[theorem]{Problem}
\newtheorem{proposition}[theorem]{Proposition}

\input{tcilatex}
\newcommand{\arrow}{\rightarrow}

\newcommand{\T}{{\mathbb{T}}}
\newcommand{\lgr}{\underline{\rm{gr}}\,}
\newcommand{\cluster}{\mathcal C}
\newcommand{\eps}{\varepsilon}
\newcommand{\one}{\mathbf 1}
\newcommand{\ev}{{\mathbf{E}}}
\newcommand{\pr}{{\mathbf{P}}}
\newcommand{\tr}{{\sf{tr\,}}}
\newcommand{\ZZ}{{\mathbb Z}}
\renewcommand{\emph}[1]{{\bf #1}}
\endlocaldefs

\endlocaldefs

\begin{document}
\begin{frontmatter}
\title{The measurable Kesten theorem }
\runtitle{Measurable Kesten}

\begin{aug}
\author{\fnms{Mikl\'{o}s} \snm{Ab\'{e}rt} \thanksref{m1} \ead[label=e1]{abert@renyi.hu}},
\author{\fnms{Yair} \snm{Glasner} \thanksref{m2} \ead[label=e2]{yairgl@math.bgu.ac.il}}
\and
\author{\fnms{B\'alint} \snm{Vir\'ag} \thanksref{m3} \ead[label=e3]{balint@math.toronto.edu}}

 \affiliation{Alfr\'ed R\'enyi Institute of Mathematics \thanksmark{m1}, Ben-Gurion University of the Negev \thanksmark{m2} and University of Toronto \thanksmark{m3} }

 \address{Alfr\'ed R\'enyi Institute of Mathematics. \\
                Re\'altanoda utca 13-15, H-1053, \\  Budapest,  Hungary. \\
                \printead{e1} \\
	        \phantom{E-mail:\ }}

\address{Department of Mathematics. \\
		Ben-Gurion University of the Negev, \\
		P.O.B. 653, \\
		Be'er Sheva 84105, Israel. \\
		\printead{e2} \\
		\phantom{E-mail:\ }}

\address{Departments of Mathematics and Statistics. \\
		University of Toronto, M5S 2E4 \\
		Canada.  \\
		\printead{e3} \\
\phantom{E-mail:\ }}
 \end{aug}

\begin{abstract}
We give an explicit bound on the spectral radius in terms of the densities of
short cycles in finite $d$-regular graphs. It follows that the
a finite $d$-regular Ramanujan graph $G$ contains a negligible number of cycles of size less than $c
\log \log \left\vert G\right\vert$.

We prove that infinite $d$-regular Ramanujan unimodular random graphs are trees. Through Benjamini-Schramm convergence this leads to the following rigidity result. If most eigenvalues of a $d$-regular finite graph $G$ fall in
the Alon-Boppana region, then the eigenvalue distribution of $G$ is close to the spectral measure of the $d$-regular tree. In particular, $G$ contains
few short cycles.

In contrast, we show that $d$-regular unimodular random graphs with maximal growth are not necessarily trees.
\end{abstract}



\end{frontmatter}
\section{Introduction}

Let $G$ be a $d$-regular, finite or infinite connected undirected graph. Let $M$ be
the Markov averaging operator on $\ell ^{2}(G)$. When $G$ is infinite, we
define the {\bf spectral radius} of $G$, denoted $\rho (G)$, to be the
norm of $M$. When $G$ is finite, we want to exclude the trivial eigenvalues
and thus define $\rho (G)$ to be the second largest element in the set of
absolute values of eigenvalues of $M$. For an infinite graph $G$, we have $%
\rho (G)\geq \rho (T_{d})=2\sqrt{d-1}/d$ where $T_{d}$ denotes the $d$%
-regular tree. For finite graphs, the Alon-Boppana theorem \cite{alon} says
that $\lim \inf \rho (G_{n})\geq \rho (T_{d})$ for any infinite sequence $%
(G_{n})$ of finite connected $d$-regular graphs with $\left\vert
G_{n}\right\vert \rightarrow \infty $.

We call $G$ a \emph{Ramanujan} graph, if $\rho (G)\leq \rho (T_{d})$.
Lubotzky, Philips and Sarnak \cite{lps}, Margulis \cite{margulis} and
Morgenstein \cite{morgen} have constructed sequences of $d$-regular
Ramanujan graphs for $d=p^{\alpha }+1$. Also, Friedman \cite{friedman}
showed that random $d$-regular graphs are close to being Ramanujan.

All the Ramanujan graph families above have \emph{large girth}, that is, the
minimal size of a cycle tends to infinity with the size of the graph.
However, the reason for that is group theoretic and not spectral, and a
priori, Ramanujan graphs could have many short cycles.

In this paper we investigate the connection between the densities of short
cycles, the spectral radius and the spectral measure for $d$-regular graphs.
We apply our methods to give explicit
estimate these invariants, then we pass to graph limits and prove limiting results.

\subsection{Explicit estimates}

A cycle (or $k$-cycle) in a graph is a walk of length $k$ that starts and ends at the same vertex.
It is called {\bf nontrivial} if either for some directed non-loop edge $e$,
the number of times the cycle passes
through $e$ differs from the number of
times it passes through the reversal of $e$, or $k=1$ (see Definition \ref{d:trivial}).
For a finite graph $G$ let $\gamma_k (G)$ denote the number of
nontrivial $k$-cycles in $G$ divided by the number of vertices $|G|$ of $G$.

\begin{theorem}
\label{vegesbecsles}Let $G$ be a finite $d$-regular graph with $|G|\ge d^7$. Then for any $%
k\geq 1$ we have
\begin{equation*}
\frac{\rho (G)}{\rho ({\mathbb{T}}_{d})}\geq 1+\frac{\gamma_k (G)}{\nu_{k}}-%
\frac{\frac{3}{2}\log \log _{d-1}|G|+6}{\log _{d-1}|G|}
\end{equation*}%
where
$\nu_k=2\cdot 10^{11} 2^{4k}(d-1)^{3k}k.$
\end{theorem}

Applying this to finite Ramanujan graphs yields that they have few cycles of length $o(\log \log |G|)$.

\begin{theorem}
\label{essgirth}Let $d\geq 3$ and $\beta=(30 \log(d-1))^{-1}$. Then for any $d$-regular finite Ramanujan
graph $G$, the proportion of vertices in $G$ whose $\beta\log \log |G|$%
-neighborhood is a $d$-regular tree is at least $1-c(\log|G|)^{-\beta}$.
\end{theorem}

This answers a question of Lubotzky \cite[Question 10.7.1]{lubbook} who
asked for a clarification on the connection between eigenvalues and girth.
Note that until now, it was not even known whether a finite
Ramanujan graph cannot have a positive density of short cycles.

It is easy to see that infinite Ramanujan graphs can have arbitrarily many
short cycles. In fact, every connected, infinite $d$-regular graph can be
embedded as a subgraph of a Ramanujan graph with degree at most $d^{2}$ (see
Corollary \ref{c:inf_ram_criterion}). However, it turns out that cycles of bounded size must be
sparse in a Ramanujan graph.

\begin{theorem}
\label{distance}Let $G$ be an infinite $d$-regular graph such that every
vertex in $G$ has distance at most $R$ from a $k$-cycle. Then
\begin{equation*}
\rho (G)\geq \rho (T_{d})+\frac{d-2}{d(d-1)^{2\left\lfloor
R+k/2+1\right\rfloor }}\text{.}
\end{equation*}
\end{theorem}

\subsection{Graph limits and spectral measure}

The spectral measure $\mu _{T_{d}}$ of the Markov operator on $T_{d}$, also
known as the Plancherel measure of $T_{d}$ or the Kesten-McKay measure, has
density
\begin{equation*}
\frac{d}{2\pi }\frac{\sqrt{\rho ^{2}(T_{d})-t^{2}}}{1-t^{2}}\text{.}
\end{equation*}%
Let $(G_{n})$ be a sequence of finite $d$-regular graphs. We say that $%
(G_{n})$ has \emph{essentially large girth}, if for all $k$ the denisty of nontrivial cycles satisfies
\begin{equation*}
\lim_{n\rightarrow \infty }{\gamma_{k}(G_{n})}
=0.
\end{equation*}%

For a finite graph $G$, let $\mu_{G}$ denote the eigenvalue distribution of the Markov operator on $G$.
Then the following are equivalent (see Proposition \ref{ekvivalens}):

\begin{enumerate}
\item $(G_{n})$ has essentially large girth;

\item $(G_{n})$ converges to $T_{d}$ in Benjamini-Schramm convergence;

\item $\mu _{G_{n}}$ weakly converges to $\mu _{T_{d}}$.
\end{enumerate}

A sequence $(G_{n})$ of finite $d$-regular graphs is \emph{weakly Ramanujan} if
\begin{equation*}
\lim_{n\rightarrow \infty }\mu _{G_{n}}\left( [-\rho (T_{d}),\rho
(T_{d})]\right) =1\text{,}
\end{equation*}%
that is, if most eigenvalues of $G_{n}$ fall in the minimal possible
supporting region. Note that a weakly Ramanujan sequence is not necessarily
an expander sequence. In fact, the graphs $G_{n}$ do not even have to be
connected.

From 1) $\Longrightarrow $ 3) and the fact that $\mu _{T_{d}}$ is
continuous, it follows immediately that every graph sequence of essentially
large girth is weakly Ramanujan (in contrast, $\rho$ is only lower semicontinuous with respect to
Benjamini-Schramm convergence of graphs). We show that the converse also holds.

\begin{theorem}
\label{weakram}Let $(G_{n})$ be a weakly Ramanujan sequence of finite $d$%
-regular graphs. Then $(G_{n})$ has essentially large girth.
\end{theorem}

Theorem \ref{weakram} can also be looked at as a rigidity result, as it says
that if we force most of the eigenvalues of the Markov operator of a large
finite graph inside the Alon-Boppana bound, then their distribution will be
close to $\mu _{T_{d}}$.

In the proof of Theorem \ref{weakram}, it is the use of Benjamini-Schramm
convergence that allows us to get rid of the bad eigenvalues and clear up
the picture. Limit objects with respect to this convergence are random rooted graphs $(G,o)$ called
unimodular random graphs. We will sometimes drop the root $o$ from the notation.
The notion has been introduced in \cite%
{aldouslyons}: for the definition, see Section 2. Unimodular random graphs
tend to behave like vertex transitive graphs in many senses. Theorem \ref%
{weakram} now follows from the following.

\begin{theorem}
\label{unimod}Let $(G,o)$ be a $d$-regular unimodular random graph that is
infinite and Ramanujan a.s. Then $G=T_{d}$ a.s.
\end{theorem}

This is Kesten's theorem for vertex transitive graphs (\cite{kesten1} and \cite{paschke}).
We give the following two quantitative versions of Theorem \ref{unimod}. For infinite $d$-regular unimodular random graphs
\begin{equation}\label{e:urg bounds}
\ev \log \rho (G)-\log \rho ({\mathbb{T}}_{d})\geq \begin{cases} \frac{1}{\nu_k}\ev \gamma_k(G,o)  \\-\frac{1}{k}\mathbf{E} \log
\kappa^*_{k}(G,o).
\end{cases}
\end{equation}
Here  $\gamma_k(G,o)$ denotes
the number of nontrivial $k$-cycles starting at $o$, and $\nu_k$ is a constant defined in Theorem \ref{vegesbecsles}. Note that for a fixed finite graph $G$ the density $\gamma_k(G)$ equals the expected value of $\gamma_k(G,o)$ over a uniformly chosen root $o$ of $G$.

To define $\kappa^*_k(G,o)$, consider all paths of length $k$ from
$o$ to a vertex $v$. After attaching a fixed path from $v$ to $o$, these
can be used as generators for a random walk on the fundamental group of $G$.
Then $\kappa^*_k(G,o)$ is the geometric average of the
spectral radii of these random walks when $v$ is a chosen randomly as the position of the
infinite nullcycle  (defined in Corollary
\ref{c:inifinite nullcycle}) at time $k$ (see \eqref{e:kappa}, \eqref{e:kappa*} for more details).

Note that if our unimodular random graph $G$ is not a tree, then for $k$ large enough, with positive probability,
the Cayley graph of the subgroup of the fundamental group given by above loops as generators has spectral radius less than one. Thus the second bound clearly implies Theorem \ref{unimod}.

The first bound in \eqref{e:urg bounds} is proved in Section \ref{s:explicit bounds on rho} as Theorem \ref{t:main}; the proof uses results from Sections \ref{s:Cycles in Td} and  \ref{s:nullcycles}. It is just the infinite version of Theorem \ref{vegesbecsles}. The advantage of
this approach is the linear lower estimate on how the spectral radius grows
compared to the tree: we believe this to be sharp. The major advantage of the second bound in \eqref{e:urg bounds} (proved in Section \ref{s:rho and pi1}) is that it is
sharp in limit as $k\to\infty$, however, $\kappa_k^*$ seems to be hard to
compute. 


Theorem \ref{weakram} is related to a paper of Serre \cite{serre}
that studies asymptotic properties of graph sequences. Let $d_{k}(G)$ denote
the number of primitive, cyclically reduced cycles of length $k$ in the
graph $G$. Recall that a cycle is primitive if it is not a proper power of
another cycle.

\begin{theorem}[Serre]
Let $(G_{n})$ be a sequence of finite $d$-regular graphs, such that the
limit
\begin{equation*}
\gamma' _{k}=\lim_{n\rightarrow \infty }d_{k}(G_{n})/|G_{n}|
\end{equation*}
exists for every $k$. Then the measures $\mu _{G}$ weakly converge. If the
series
\begin{equation*}
\sum_{k=1}^{\infty }\gamma' _{k}(d-1)^{-k/2}
\end{equation*}
converges then the sequence of graphs is weakly Ramanujan and the limiting
measure is absolutely continuous with respect to the Lebesgue measure on $%
\left[ -\rho (T_{d}),\rho (T_{d})\right] $.
\end{theorem}

Theorem \ref{weakram} now immediately implies the following.

\begin{corollary}
If the series
\begin{equation*}
\sum_{k=1}^{\infty }\gamma' _{k}(d-1)^{-k/2}
\end{equation*}%
converges, then $\gamma' _{k}=0$ for all $k$ and the limiting measure of $\mu_{G_{n}}$ equals $\mu _{T_{d}}$.
\end{corollary}

It is natural to ask whether a version of Theorem \ref{unimod} holds for growth instead
of spectral radius. In Section \ref{s:growth} we show that the answer is negative:

\begin{theorem}
\label{growth}There exists an infinite $d$-regular unimodular random graph
with the same growth as $T_{d}$ but not equal to $T_{d}$.
\end{theorem}

We obtain our example by considering the universal cover of the infinite
cluster in supercritical percolation over $\mathbb{Z}^{2}$.

\subsection{The basic method} \label{bm}

There is a common method in the proofs of Theorems \ref{vegesbecsles} and \ref{unimod} which we can summarize as follows.

The central tool of our analysis is a nullcycle. Recall that a cycle is a walk of finite length that starts and ends at the same vertex.

\begin{definition}\label{d:nullcycle} A {\bf nullcycle} is
a cycle in a graph $G$ so that if we keep deleting backtrackings (steps that are immediately reversed), we get a cycle of length 0.
\end{definition}

This property does not depend on the order of erasing backtrackings.
Equivalently, a nullcycle is a cycle whose lift in the universal covering tree of
the graph is also a cycle. In other words, the walk corresponds to a trivial element in the fundamental group of the graph $G$. The number of nullcycles in a $d$-regular graph starting at a fixed vertex $v$ equals the number of cycles in the $d$-regular tree at a fixed vertex.

To bound the spectral radius, we have to count cycles of a given length. In order to bound the spectral radius away from that of the tree, we need to show that there are exponentially more cycles than nullcycles. Consider the set of cycles of length $nk$ starting at a vertex $v$ in a $d$-regular graph $G$. We say that $w'$ is a {\bf rewiring} of $w$ if they are
at the same place at times that are multiples of $k$. This definition is used in Section \ref{s:rho and pi1}; in Section
\ref{c and n} we use a slight variant of this.

Consider the equivalence class $[w]$ of a typical nullcycle $w$ under the rewiring equivalence relation. The essence of our argument is to show that for a typical $w$, the probability that a random element $C$ of $[w]$ is a nullcycle is exponentially small.
Essentially, in every segment $[jk, (j+1)k]$, if there are short cycles around in the graph, there is a positive probability that the rewiring $C$ will use them, and this is likely to stop  $C$ from being null-homotopic.

In order to show that $C$ is nullhomotopic with exponentially small probability, we need to find a linear number of $j$ so that $G$ has short cycles around $w(jk)$. Fortunately,  the random nullcycle $w$ samples the graph $G$ in a homogeneous manner. In particular, if $w(0)$ is  a uniformly chosen vertex, then so will be $w(j)$ for every $j$. This is an advantage of using random nullcycles over random cycles. For infinite graphs, the proof of this step uses unimodularity.

A crucial property that we use to get explicit bounds is one that random nullcycles share with simple random walks. Let $G$ be a $d$-regular rooted graph and let $W$ be a
uniform random nullcycle of length $o(\sqrt{|G|})$, starting at the root.
Then the expected number of visits of $W$ at any vertex of $G$ can be bounded above in terms of $\rho (G)$ (without referring to the length of the cycle). In particular, for a good
expander graph, the expected number of returns of a random nullcycle is
bounded. We need this property to show that a typical rewiring will not use the same short cycles over and over again. This is a technically difficult point that we tackle in Section \ref{s:nullcycles}.

Putting all these together, we get that if there are many short cycles, then a typical nullcycle  will get close to short cycles at linearly many different times. Thus a random rewiring will be a nullcycle only with exponentially small probability. In other words, there are exponentially more cycles than nullcycles, which implies that the spectral radius of $G$ is greater than that of the tree $\T_d$.

%

%

\subsection{Open problems}

It is not clear whether the $\log \log |G|$ is optimal in
Theorem \ref{essgirth}. For all the known examples of graphs that are close
to being Ramanujan, the shortest cycles with positive density are actually logarithmic.

\begin{problem} \label{pp}
Is there a constant $c=c(d)>0$ such that for any $d$-regular Ramanujan graph
sequence $(G_{n})$, the probability that the $c\log \left\vert
G_{n}\right\vert $-neighborhood of a uniform random vertex in $G_{n}$ is a
tree converges to $1$?
\end{problem}

A standard ergodicity argument says that for an ergodic unimodular random
graph $G$, the weak limit of the random walk neighborhood sampling of $G$
gives back the distribution of $G$ a.s. \cite{benscha}. This suggests the
following possible generalization of Theorem \ref{unimod}.

\begin{problem}
\label{kerdes1}Let $G$ be an infinite $d$-regular rooted Ramanujan graph and
let $k>0$. Let $p_{n}$ denote the probability that the random walk of length
$n$ on $G$ ends on a $k$-cycle. Is it true that $p_{n}$ converges to $0$?
\end{problem}

That is, is it true that the random walk neighborhood sampling of $G$
converges to $T_{d}$? The answer does not follow from Theorem \ref{unimod},
even when the random walk sampling converges, as the limit is only a
stationary distribution on rooted graphs and is not necessarily unimodular.
It would also be interesting to see whether Theorem \ref{unimod} holds for
stationary random graphs. The recent paper \cite{grikanag} solves
Problem \ref{kerdes1} affirmatively in the case when the so-called co-growth
of $G$, the exponent of the probability of return for a non-backtracking
random walk, is less than $\,1/\sqrt{d-1}$. However, when the co-growth
equals $1/\sqrt{d-1}$, the graph is still Ramanujan but the answer seems
unclear. We thank Tatiana Smirnova-Nagnibeda for communicating this with us. After the first preprint version of this paper appeared, R. Lyons and Y. Peres, \cite{lyons_peres} generalized our results and in particular gave a positive answer to Problem \ref{kerdes1}.

The linear lower
estimate in the spectral radius in Theorem \ref{vegesbecsles} seems to be
sharp, but we have not been able to settle this with a suitable family of
examples. The same is true for unimodular random graphs (see the first bound of \eqref{e:urg bounds}).

\begin{problem}
Does there exists $C>0$ such that for every $r>0$ there exists an infinite $%
d $-regular unimodular random graph $G$ with
\begin{equation*}
\rho (G)\leq \rho (T_{d})+Cr
\end{equation*}%
such that the density of loops in $G$ is at least $r$?
\end{problem}

One natural idea would be to use a modified universal cover of a finite $d$%
-regular graph of size $n$ with a loop, where we never open the loop in the
cover. It looks reasonable that this cover (which is a finitely supported
random rooted tree with loops) should have spectral radius around $\rho (T_{d})+C/n$%
.

\bigskip

The paper is organized as follows. Section \ref{section_preliminaries}
contains the basic definitions and we prove some lemmas that will be used
throughout the paper. In Sections \ref{s:Cycles in Td} and \ref{s:nullcycles}
we use properties of cycles in trees to study nullcycles, which are
needed for Theorem \ref{vegesbecsles}.
In Section \ref{s:explicit bounds on rho} we prove
Theorem \ref{t:main}, a more general version of Theorems \ref{vegesbecsles} and \ref{unimod}.
We also show
Theorem \ref{t:ess girth r graphs}, a more general version of Theorem \ref{essgirth}.
Finally, in this Section we also prove Theorem \ref{weakram}.

Section \ref{s:rho and pi1} contains a sharp bound
on the spectral radius in terms of random walks on the fundamental group.
Section \ref{section_dense}
contains the proof of Theorem \ref{distance}. This section is independent of
the rest.

In Section \ref{s:fin and infite rg}
 we give example of Ramanujan graphs with loops, and Section \ref{s:growth} we prove Theorem \ref{growth}.

Note that one can read Section \ref{s:explicit bounds on rho}, and \ref{s:rho and pi1}
independently, after reading Section \ref{section_preliminaries}, but reading any of these two will give help when
reading the other.

An earlier version of this paper contained a generalization of Kesten's theorem on
groups. As the readership of this result is expected to be different from
that of the current paper (and the current paper is already long),
we decided to publish it in a separate article, see \cite{kestengroup}.

\section{Preliminaries \label{section_preliminaries}}

In this section we define the notions and state some basic results used in
the paper.\bigskip

We follow Serre's notation for a graph, with a modification on how to define
loops. A graph $G$ consists of two sets, a set of vertices denoted by $V(G)$
and a set of edges denoted $E(G)$. For every edge $e\in E(G)$ there are
vertices $e^{-}$ (the initial vertex) and $e^{+}$ (the terminal vertex). We
allow $e^{-}=e^{+}$: such edge is called a \emph{loop}. For every edge $e$
there is a \emph{reverse edge} $\overline{e}\in E(G)$ such that $\overline{e}%
^{+}=e^{-}$ and $\overline{e}^{-}=e^{+}$. For a loop $e$, we allow $%
\overline{e}=e$; these are called \emph{half-loops}. The degree of a vertex $%
v$ is
\begin{equation*}
\deg v=\left\vert \left\{ e\in E(G)\mid e^{-}=v\right\} \right\vert
\end{equation*}%
So half-loops contribute $1$ to the degree, but loops together with their
distinct inverse contribute $2$. For spectral and random walk questions, each (non-half) loop
can be replaced by two half-loops. So in this paper we will assume that all loops are half-loops.

A graph is $d$-regular, if all vertices
have degree $d$.
\bigskip

%
A \emph{walk} of length $n$ is a sequence of directed edges $%
w=(w_{1},w_{2},\ldots ,w_{n})$ such that $w_{i-1}^{+}=w_{i}^{-}$ ($2\leq
i\leq n$). The walk is a \emph{cycle} if $w_{1}^{-}=w_{n}^{+}$. The \emph{%
vertices} of the walk are defined by $w(i-1)=w_{i}^{-}$ and $w(n)=w_{n}^{+}$
is the \emph{end of the walk}. The \emph{inverse} of a walk $w$ is defined
by $w^{-1}=(\overline{w_{n}},\overline{w_{n-1}},\ldots ,\overline{w_{1}})$.
A cycle is a \emph{nullcycle} if its lift to the universal cover of $G$
stays a cycle. That is the same as saying that if we keep erasing backtracks
from the cycle, we get to the empty walk.
For a rooted graph $(G,o)$  we will denote the set of nullcycles of length $n$ by $\mathcal{N}_n$.

For a graph $G$ and $x,y\in V(G)$ let $W_n(x,y)$ denote the set of walks of
length $n$ starting at $x$ and ending at $y$. A \emph{random walk of length }$n$\emph{\ starting
at }$x$ is a uniform random walk starting at $x$. Let $p_n(x,y)$ denote the probability that
a random walk of length $n$ started at $x$ ends at $y$.  We call $p_n(x,x)$ the $n$-step
{\bf return probability}.

Let $G$ be a $d$-regular, connected undirected graph. Let $\ell
^{2}(G)$ be the Hilbert space of all square summable functions on the vertex
set of $G$. Let us define the Markov operator $M:\ell ^{2}\rightarrow \ell
^{2}$ as follows:
\begin{equation*}
(Mf)(x)=\frac{1}{d}\dsum\limits_{e\in E(G), e^{-}=x } f(e^{+})
\end{equation*}

When $G$ is infinite, we define the {\bf spectral radius} of $G$, denoted
$\rho (G)$, to be the norm of $M$. When $G$ is finite, we want to exclude
the trivial eigenvalues and thus define $\rho (G)$ to be the second largest
element in the set of absolute values of eigenvalues of $M$. Note that when the connected graph $G$ is bipartitie, then
$-d$ is an eigenvalue with multiplicity one; this is not counted in the definition of $\rho(G)$.

In the case when $G$ is infinite and connected, one can express the spectral
radius of $G$ from the return probabilities as follows:
\begin{equation*}
\rho (G)=\lim_{n\rightarrow \infty }p_{2n}(x,x)^{1/2n}
\end{equation*}%
where $x$ is an arbitrary vertex of $G$.

The Markov operator $M$ is self-adjoint, so we can consider its spectral
measure. This is a projection valued measure $P$ such that $%
P(O):l^{2}(G)\rightarrow l^{2}(G)$ is a projection for every Borel set $%
O\subset \lbrack -1,1]$. For every $f\in l^{2}(G)$ with $\left\Vert
f\right\Vert _{2}=1$, the expression
\begin{equation*}
\mu _{f}(O)=\langle P(O)f,f\rangle
\end{equation*}%
defines a Borel probability measure on $[-1,1]$.

For graph $G$ rooted at $o$, let the \emph{spectral measure of }$G$ be
\begin{equation*}
\mu _{G,o}=\mu _{\delta _{o}}
\end{equation*}%
where $\delta _{o}\in l^{2}(G)$ is the indicator function of $o$. The
best way to visualize this measure is to look at its moments, that satisfy
the following equality:
\begin{equation*}
\dint\limits_{\lbrack -1,1]}x^{k}d\mu _{G,o}=p_k(o,o)
\end{equation*}%
for all integers $k\geq 0$. \bigskip

\subsection*{Unimodular random graphs} Heuristically, a unimodular
random graph is a probability distribution on rooted graphs that stays
invariant under moving the root to any direction. However, one has to be
careful with this intuition, as direction is not well-defined and indeed,
there exist vertex transitive graphs that we want to exclude from the
definition. We follow \cite[Section 5.2]{aldoussteele} in our definition
restricted to the $d$-regular case where it is somewhat simpler.

A \emph{flagged graph} is a graph with a distinguished root and a directed
edge starting at the root. One can \emph{invert the flag} by moving the root
to the other end of the flag and switching the direction of the flag.

\begin{definition}
Let $G$ be a probability distribution on rooted $d$-regular graphs. Pick a
uniform random edge from the root and put a flag on it. This gives a
probability distribution $\widetilde{G}$ on flagged $d$-regular graphs. We
say that $G$ is a \emph{unimodular random graph}, if the distribution $%
\widetilde{G}$ stays invariant under inverting the flag.
\end{definition}

That is, if some of the flagged lifts of a given rooted graph are
isomorphic, we count it with multiplicity. Note that vertex transitivity in
itself does not imply unimodularity. A simple example is the so-called
grandmother graph. This can be obtained by taking a $3$-regular tree and directing it towards a boundary point, then connecting every vertex to the ascendant of its ascendant (its grandmother) and then erasing directions (see Figure \ref{fig:grandmother}).

If one does not mind working with edge directed graphs, it is easier to see the lack of unimodularity in the oriented 3-regular tree itself. There is only one type of rooted graph here that obviously appears with probability $1$. The corresponding measure on flagged graphs puts the flag on an outgoing edge with probability $1/3$, but after an inversion we see an outgoing edge with probability $2/3$.
See \cite{aldouslyons} for more about unimodularity.

\begin{figure}[ht]
\begin{center}
\includegraphics[width=5cm]{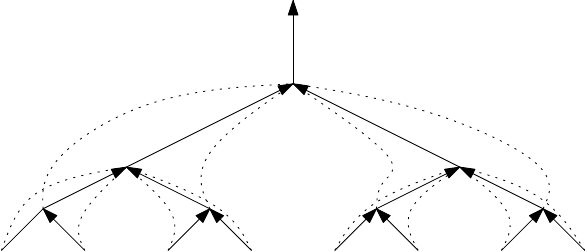}
\end{center}
\caption{The grandmother graph}
\label{fig:grandmother}
\end{figure}

\subsection*{Mass Transport Principle} The most useful property about
unimodular random graphs (that can also be used to define them) is the Mass
Transport Principle which is as follows. Let $f$ be a non-negative
real-valued function on triples $(G,x,y)$ where $G$ is a $d$-regular rooted
graph and $x,y\in G$ such that $f$ does not depend on the location of the
root. Then the expectations
\begin{equation*}
\mathbf{E}\left[ \dsum\limits_{y\in G}f(G,o,y)\right] =\mathbf{E}\left[
\dsum\limits_{x\in G}f(G,x,o)\right]
\end{equation*}%
where $o$ is the root of $G$. The picture is that if one sets up a paying
scheme on the random graph $G$ that is invariant under moving the root, then
the expected payout of the root equals its expected income.

\subsection*{Benjamini-Schramm convergence} A $d$-regular \emph{graph
sequence} $(G_{n})$ is defined as a sequence of finite $d$-regular graphs
with size tending to infinity. By a \emph{pattern of radius }$r$ we mean a
rooted graph where every vertex has distance at most $r$ from the root. For
a finite graph $G$ and a pattern $\alpha $ of radius $r$ let the sampling
probability $p(G,\alpha )$ be the probability that the $r$-ball around a
uniform random vertex of $G$ is isomorphic to $\alpha $. We say that a graph
sequence $(G_{n})$ is Benjamini-Schramm convergent, if $p(G_{n},\alpha )$ is
convergent for every pattern $\alpha $. It is easy to see that every graph
sequence has a convergent subsequence.

What is a natural limit object of a convergent graph sequence? One can also
take pattern densities of a unimodular random graph $G$; there $p(G,\alpha )$
denotes the probability that the $r$-ball around the root of $G$ is
isomorphic to $\alpha $. We say that a graph sequence $(G_{n})$ converges to
$G$ if
\begin{equation*}
\lim_{n\rightarrow \infty }p(G_{n},\alpha )=p(G,\alpha )\text{ for all
patterns }\alpha \text{.}
\end{equation*}%
Every Benjamini-Schramm convergent graph sequence has a unique limit
unimodular random graph (see \cite[Section 2.4]{aldoussteele}).

For a finite $d$-regular graph $G$ let $\mu _{G}$ denote the eigenvalue
distribution of the Markov operator on $G$. Note that for a uniform random vertex $o$ we have $\mu_G=\ev \mu_{G,o}$. For an infinite unimodular random graph $G$ we can also define $\mu_G=\ev \mu_{G,o}$.

\begin{proposition}
\label{ekvivalens} Let $(G_{n})$ be a sequence of finite $d$-regular graphs.
Then the following are equivalent: \newline
1) $(G_{n})$ has essentially large girth; \newline
2) $(G_{n})$ converges to $T_{d}$ in Benjamini-Schramm convergence; \newline
3) $\mu _{G_{n}}$ weakly converges to $\mu _{T_{d}}$.
\end{proposition}

\begin{proof} The equivalence of 1) and 2) is immediate from the
definition of Benjamini-Schramm convergence.

Assume that $(G_{n})$ converges to the unimodular random graph $G$. We claim
that $\mu _{G_{n}}$ weakly converges to the expected spectral measure $\mu_G =%
\mathbf{E} \mu _{G,o}$. To check this, we can look at the $k$th moment
\begin{equation*}
\dint x^{k}d\mu_G =\mathbf{E}\left[ p_k^G(o,o)\right] \text{.}
\end{equation*}%
Recall that $p^G_k(o,o)$ denotes the probability of return of the random
walk on $G$ starting at $o$. But for any graph $G$ and vertex $v$ of $G$,
the return probability $p^G_k(v,v)$ only depends on the $k/2$-ball around $o$%
. Since there are only finitely many patterns of a given radius, this
implies
\begin{equation*}
\mathbf{E}\left[ p^G_k(o,o)\right] =\dsum\limits_{\alpha \text{ is a pattern
of radius }\lfloor k/2\rfloor }p(G,\alpha )p^\alpha_k(v,v)
\end{equation*}%
where $v$ is the root of $\alpha $. Now $(G_{n})$ converges to $G$, so
\begin{eqnarray*}
\mathbf{E}\left[ p_k^G(o,o)\right] &=&\lim_{n\rightarrow \infty
}\dsum\limits_{\alpha \text{ is a pattern of radius }\lfloor k/2 \rfloor}p(G_{n},\alpha
)p^\alpha_k(v,v)= \\
&=&\lim_{n\rightarrow \infty
}\mathbf{E}\left[ p^{G_{n}}_k(u,u)\right] =\lim_{n\rightarrow \infty
}\dint x^{k}d\mu _{G_{n}}
\end{eqnarray*}%
where $u$ is a uniform random vertex in $G_{n}$. So, $\mu _{G_{n}}$ weakly
converges to $\mu_G $ as claimed. Hence 2) implies 3)\ follows immediately.

Assume that 1) does not hold, that is, $(G_{n})$ is a graph sequence that
does not have essentially large girth. Then there exists $k,\varepsilon >0$
such that the density of $k$-cycles in $G_{n}$ is at least $\varepsilon $
for infinitely many of the $G_{n}$. This implies that for these $n$,
\begin{equation*}
\dint x^{k}d\mu _{G_{n}}=\mathbf{E}\left[ p^{G_{n}}_k(u,u)\right] \geq
p^{T_{d}}_k(o,o)+\frac{\varepsilon }{d^{k}}=\dint x^{k}d\mu _{T_{d}}+\frac{%
\varepsilon }{d^{k}}
\end{equation*}%
which implies that $\mu _{G_{n}}$ does not converge weakly to $\mu _{T_{d}}$%
. Hence, 3) does not hold.
We proved the required equivalences. \end{proof}

\subsection*{Fundamental group} Let $G$ be a graph rooted at $o$. We
call two cycle starting at $o$ \emph{homotopic}, if one can get one
from the other by inserting and erasing backtracks, that is, walks of type $s%
\overline{s}$ where $s$ is an edge of $G$. Then the set of equivalence
classes forms a group under concatenation, called the \emph{fundamental group%
} $\pi _{1}(G)$. It is well known that the fundamental group of a graph
without half-loops is a free group \cite[Theorem 5.1]{massey}. Every
half-loop adds a cyclic group of order $2$ as a free product. The most
important general property of fundamental groups we shall use in this paper
is that if $H$ is a subgraph of $G$, then the induced homomorphism from $\pi
_{1}(H)$ to $\pi _{1}(G)$ is injective.

\section{Cycles in $\T_d$}
\label{s:Cycles in Td}

This section establishes some basic properties of $\mathcal N_n=\mathcal N_n(d)$, the set of $n$-cycles in $\T_d$.
Such a cycle $\alpha \in \mathcal N_n$ in the $3$-regular tree is depicted in Figure \ref{fig:cycle}. Given any covering map $p:T_3 \arrow X$ to a $3$-regular graph $X$, the
projection of the cycle $p(\alpha)$ is referred to as a null cycle in the graph $X$.  \begin{figure}[ht]
\begin{center}
\includegraphics[width=5cm]{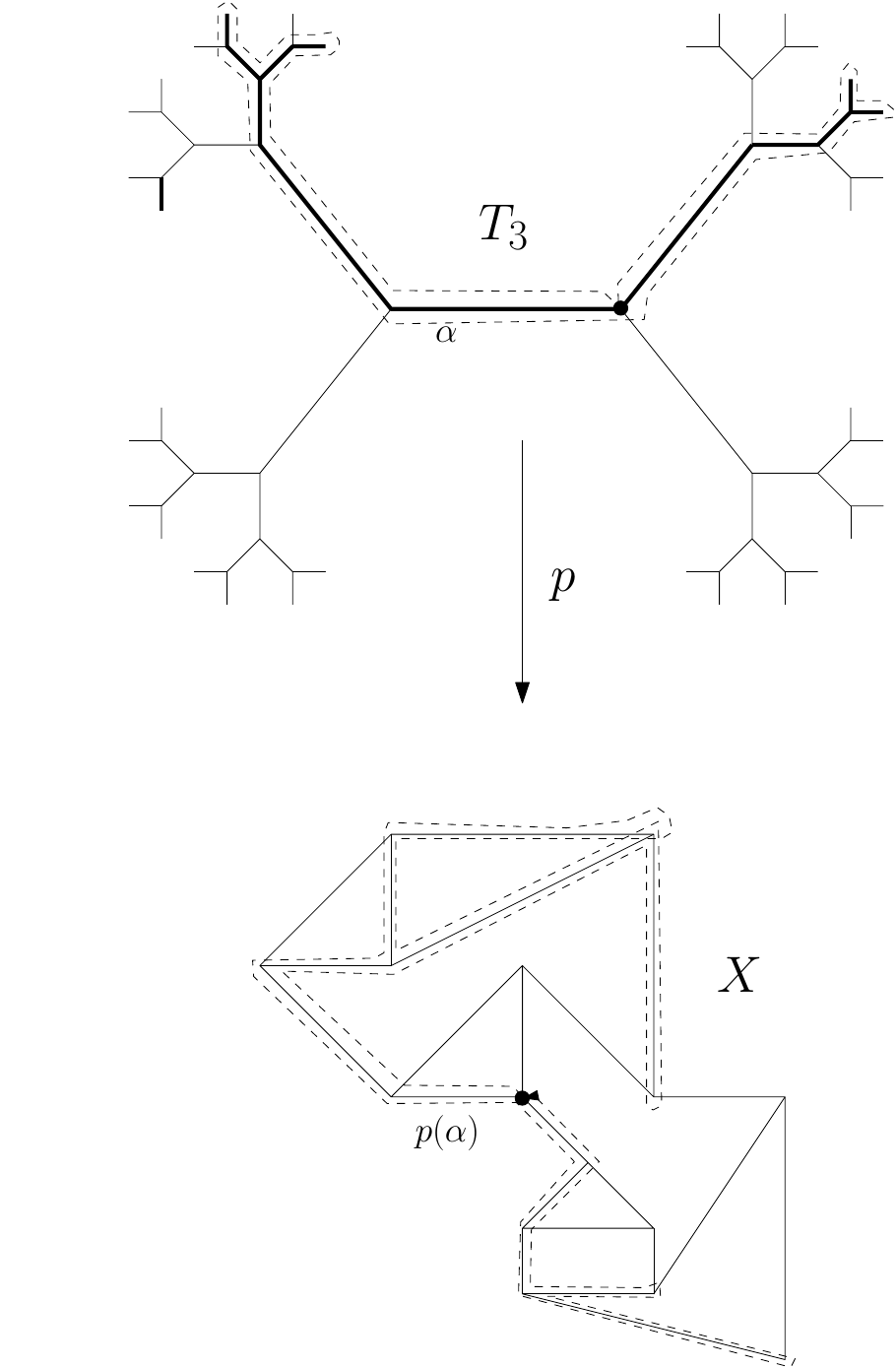}
\end{center}
\caption{A cycle in the $3$-regular tree}
\label{fig:cycle}
\end{figure}

\subsection{Explicit return probability bounds}
We start by estimating the size of $\mathcal{N}_n$.
\begin{lemma}[Return probabilities of SRW on $\T_d$]\label{l:returns}
Let $\rho=\rho(\T_d)=2\sqrt{d-1}/d$.
The $n$-step return probability $r_n=d^{-n}|\mathcal N_n(d)|$ for simple random walk in $\T_d$ for even $n > 0$ satisfies
$$
\frac23 \frac{\rho^{n} }{n^{3/2} }< r_n  < 10 \frac{\rho^{n} }{n^{3/2} }.
$$
\end{lemma}
\begin{proof}
Return probabilities are moments of the spectral measure. The spectral measure in $\T_d$ is supported on $[-\rho,\rho]$ with density given by
$$
\frac{d}{2\pi} \frac{\sqrt{\rho^2-t^2}}{1-t^2},
$$
see \cite{woess2000random}, formula (19.27). So for even $n$, by symmetry, we may write
$$
r_n = \frac{d}{\pi} \int_{0}^\rho t^n\frac{\sqrt{\rho^2-t^2}}{1-t^2}\,dt=\frac{d}{2\pi} \int_{0}^{\rho^2} s^{(n-1)/2}\frac{\sqrt{\rho^2-s}}{1-s}\,ds.
$$
Then, with
$$
a=\rho^{-2}\int_0^{\rho^2} s^{(n-1)/2} \sqrt    {\rho^2-s}\, ds = \frac{\sqrt{\pi}}{2} \rho^{n} \frac {\Gamma(n/2+1/2)}{\Gamma(n/2+2)}
$$
we have
$$
\frac{d\rho^2}{2\pi}\;a<r_n <
\frac{d\rho^2}{2\pi(1-\rho^2)}\;a.
$$
A small computation shows that for $d\ge 3$ we have
$$
 \frac{8}3 \le 4-\frac 4 d= d\rho^2, \qquad \frac{d\rho^2}{1-\rho^2}= \frac{4d^2-4d}{(d-2)^2}\le 24.
$$
Now for $n\ge 4$ we have
$$
\kappa n^{-3/2}\le \frac{\Gamma(n/2+1/2)}{\Gamma(n/2+2)} \le 2^{3/2}n^{-3/2}, \qquad \kappa=4^{3/2}\frac{\Gamma(2.5)}{\Gamma(4)}.
$$
The upper bound also holds for $n=2$. (We manually check that the lower bound of the lemma holds for $r_2=1/d$.)
To complete the proof, we bound the lower and upper constants factors
\[
\frac{8}{3}\frac{1}{2\pi}\frac{\sqrt{\pi}}{2}\kappa =\frac23, \qquad 9.57\sim 24\frac{1}{2\pi}\frac{\sqrt{\pi}}{2}2^{3/2}=12\sqrt{2/\pi}< 10.
\qedhere \]
\end{proof}

Our next goal is to study the expected number of visits for random cycles in $\T_d$. This will be based on the same
question for random walk excursions on $\mathbb Z$.
Recall that an {\bf excursion} of length
$n$ on $\mathbb Z$ is a walk that stays positive except for
time $0$ and $n$, when it is zero.

\subsection{Visits of cycles}

\begin{lemma}[Counting excursions]\label{l:excursion} Let $w_{n,k}$ be the number of
walks of length $n\ge 1$ from $0$ to $k\ge 0$ in $\mathbb Z$. Then
$$
w_{n,k}< \sqrt{2/\pi} \frac{2^{n}}{\sqrt n}\;e^{-k^2/(2n)}.
$$
Let $w_{n,k}^+$ be the number of such paths that stay positive after time 0. Then for $k>0$ we have
$$
w_{n,k}^+< \sqrt{2/\pi}\frac{2^{n}k}{n^{3/2}}\;e^{-k^2/(2n)}.
$$
\end{lemma}
\begin{proof}
We may assume that $n$ and $k$ are the same parity.
Then
$$w_{n,k}=\binom{n}{\frac{n+k}2}.
$$
We use the inequality
$$
\binom{n}{\lfloor n/2\rfloor} <  \sqrt{2/\pi} \frac{2^n }{\sqrt{ n}},
$$
which holds since the ratio of the two sides is increasing along even (respectively odd) $n$
and converges to 1. For $n$ even we now write
$$
\binom{n}{\frac{n+k}2}
\binom{n}{n/2}^{-1} = \frac{((n-k)/2+1)\cdots (n/2)}{(n/2+1)\cdots ((n+k)/2)}\le \left(\frac{n-k}{n} \right)^{k/2}\le e^{-k^2/(2n)},
$$
and the odd case follows similarly.

By the Ballot theorem (see Section 2.7.1 in
\cite{LevinPeresWilmer}) we have
\[
w^+_{n,k}=\frac{k}{n}w_{n,k}
\le \sqrt{2/\pi} 2^n\frac{k}{n^{3/2}}\;e^{-k^2/(2n)}. \qedhere
\]
\end{proof}

Recall that a {\bf simple random walk excursion} of length
$n$ on $\mathbb Z$ is a uniform choice from all excursions of length $n$.
In other words, it is the simple random walk conditioned to stay positive except for
time $0$ and $n$, when it is zero.  Now we are ready to bound the expected number
of visits for simple random walk excursions on $\mathbb Z$.

\begin{lemma}[Visits of SRW excursions on $\mathbb Z$]\label{l:e visits}  The expected number of visits $v_{k,n}$ to level $k>0$ for the
simple random walk excursion of length $n$ on $\mathbb Z$
satisfies $v_{k,n}\le 64k$.
\end{lemma}

\begin{proof}
Let $w^+_{n,k}$
denote the number of walks of length $n$ starting at $0$ and ending at $k\ge 0$ that stay positive except perhaps at time $0$ and $n$.
If $X_m$ is a random walk excursion of length $n$, then
$$
v_{k,n}=\ev \sum_{m=1}^{n-1} \one(X_m=k) = \sum_{m=1}^{n-1}
\pr (X_m=k) = \frac1{w^+_{n,0}} \sum_{m=1}^{n-1}
w^+_{m,k}w^+_{n-m,k}\le  \frac2{w^+_{n,0}} \sum_{m=1}^{n/2}
w^+_{m,k}w^+_{n-m,k}
$$
For $n=2$ the claim is easy to check. For $n\ge4$ even we have the lower bound using the Catalan
number formula
$$
w^+_{n,0} = \frac{2 w_{n-2,0}}{n} \ge \frac{1}{\sqrt {2\pi}}\frac{2^n}{n^{3/2}},
$$
where the last inequality holds since the ratio of the two sides is decreasing and converges to 1.
Together with Lemma \ref{l:excursion} this gives the bound
$$
v_{k,n}\le 2 \cdot \frac{2}{\pi}k^2 \sqrt{2\pi} n^{3/2}\sum_{m=1}^{n/2}
\frac{e^{-k^2/(2m)}}{m^{3/2}(n-m)^{3/2}}\le \frac{2}{\pi} \cdot 2\cdot\sqrt{2\pi} \cdot  2^{3/2}k^2
\sum_{m=1}^{n/2}\frac{e^{-{k^2/(2m)}}}{m^{3/2}}.
$$
Let $a_m$ denote the last summand, even for non-integer $m$. Then for all $m\ge 1$ and $\delta\in [0,1]$ we have $a_{m+\delta}\ge 2^{-3/2}a_{m}$. Thus we can bound
the sum by
\[2^{3/2}\int_1^\infty \frac{e^{-k^2/(2x)}}{x^{3/2}}
\,dx<2^{3/2}\int_0^\infty \frac{e^{-k^2/(2x)}}{x^{3/2}}
\,dx= \frac{4\sqrt{\pi}}{k}. \qedhere \]
\end{proof}

A {\bf random cycle} is a cycle chosen from uniform measure from the set of cycles
with the same starting point.

\begin{lemma}[Visits of cycles in $\T_d$]\label{l:visits Td}
The expected amount of time a random cycle of even length $n$ in
$\T_d$ spends at distance $k>0$ from its starting point is at most $2\cdot10^{4}k$. For $k=0$ it is at most $301$.
\end{lemma}

\begin{proof}
 Consider a random cycle of length $n$ in $\T_d$ from the root $o$. Let $R_j$ be the distance of the walk from $o$ at time $j$. The following is well-known, see Section 2 of \cite{BougerolJeulin}.

Let $0=T_0<T_1<\cdots<T_M=n$ be the (random) times when
$R_j$ is zero. Given the values of $T_i$ and $M$, the
sections of $R_j$ in between are independent simple random walk
excursions on $\mathbb Z$. In particular, given this information, Lemma \ref{l:e visits}
implies that the conditional expectation of the number of
visits of $R_j$ to $k$ is bounded above by $64kM$. So by Lemma \ref{l:e visits} it
suffices to show that $\ev M$ is bounded by a constant
independent of $n$.

Let $r_n$ be the probability that the simple random walk on $\T_d$ visits its starting point at time $n$. By the Markov property, we have
$$
\ev M=1+\sum_{k=1}^{n/2-1} P(R_{2k}=0)=1+\frac{1}{r_n}\sum_{k=1}^{n/2-1} r_{2k}r_{n-2k}\le
1+\frac{3}{2}\cdot 10^2 \sum_{k=1}^{n/2-1}\frac{n^{3/2}}{(2k)^{3/2}(n-2k)^{3/2}}
$$
where the last inequality follows form Lemma \ref{l:returns}.
Since the summand is convex as a function of $k$, the $k$ term is bounded above by
$$
\int_{k-1/2}^{k+1/2} \frac{n^{3/2}}{(2x)^{3/2}(n-2x)^{3/2}}\,dx
$$
and the entire sum is at most
$$
\int_{1/2}^{n/2-1/2} \frac{n^{3/2}}{(2x)^{3/2}(n-2x)^{3/2}}\,dx=\frac{2 (n-2)}{\sqrt{(n-1)n} }<2
$$
This gives $\ev M < 301$.
\end{proof}

Finally, we consider the limiting process of the random cycle starting at $o$ in $\T_d$.

\begin{proposition}[The infinite cycle in $\T_d$]\label{p:inf cycle}
Let $(X^n_k,k=0\ldots n)$ be the random cycle of even length $n$ from $o$ to $o$ in $\T_d$.
Then as $n\to\infty$ the process $(X^n_k,k=0\ldots n)$ converges in distribution to a process $(X_k,k\ge 0)$ called the {\it infinite cycle}, a time-homogeneous Markov process with transition probabilities \eqref{e:inf cycle}.
\end{proposition}
\begin{proof}
The random cycle of length $n$ is a time-inhomogeneous Markov process. Let $p^n_k(x,y)$ be denote its transition probabilities from $x$ to $y$ at time $k$. It suffices to show that the ratios of $p^n_k(x,x_+)/p^n_k(x,x_-)$ converge, (where $x_+,x_-$ denotes a child or the parent of $x$, respectively) as
any probability of the form
$$
\pr(X^n_1=x_1,\ldots, X^n_k=x_k)
$$
can be written as an expression containing finitely many of these probabilities. With $p_n(x,y)$ denoting the simple random walk transition probabilities in $\T_d$, the standard path counting argument gives
$$
\frac{p^n_k(x,x_+)}{p^n_k(x,x_-)}=
\frac{p_{n-k-1}(x_+,o)}{p_{n-k-1}(x_-,o)}.
$$
we now use Theorem 19.30 in \cite{woess2000random} which for $x$ fixed and  $n\to\infty$ gives
$$p_n(o,x)=(c+o(1))\left(1+\frac{d-2}{d}|x|\right) (d-1)^{-|x|/2} \rho(\T_d)^nn^{-3/2}
$$
where $|x|$ is the graph distance of $x$ from $o$, to get
\begin{equation}\label{e:inf cycle}
\lim_{n\to\infty}\;\frac{p_{n-k-1}(x_+,o)}{p_{n-k-1}(x_-,o)} = \frac{1}{d-1} \frac{d+(d-2)(|x|+1)}{d+(d-2)(|x|-1)}=:\frac{p^*(x,x_+)}{p^*(x,x_-)}.
\end{equation}
So $(X_k, g\ge 0)$ is a time-homogeneous Markov process with transition probabilities $p^*$ (which are determined by \eqref{e:inf cycle} since they sum over the neighbors of $x$ to 1). Clearly $|X_n|$ is also a time-homogeneous Markov process, which has up/down transition probability ratio from $x\in \mathbb Z_+$ given by
\begin{equation}\notag
\frac{d+(d-2)(x+1)}{d+(d-2)(x-1)}.
\end{equation}
Note that when $d=2$ we get the reflected simple random walk, as expected.
\end{proof}

\begin{corollary}[The infinite nullcycle]\label{c:inifinite nullcycle}
Let $G$ be a $d$-regular graph, and $(\bar X^n_k,k=0\ldots n)$ be the $k^{th}$ step of a uniformly chosen random nullcycle from a vertex $o$ to $o$.
Then  $\bar X^n_k$ converges in distribution as $n\to\infty$ to a limiting process $(\bar X_k,k\ge 0)$ called the {\it infinite nullcycle}. In particular, the fixed-time distributions converge.
\end{corollary}

\begin{proof}
Note that $\bar X^n_k$ is just the image under the universal cover map from $
\T_d$ to $G$ of the random cycle in $\T_d$. So the claim follows from Proposition \ref{p:inf cycle}.
\end{proof}

\section{Properties of nullcycles}
\label{s:nullcycles}

This section establishes some important properties of
random nullcycles in graphs. But first we need a simple well-known lemma.

\begin{lemma}[Spectral radius and hitting probabilities]\label{l:kicsi}
Let $G$ be a connected $d$-regular graph and let $o$ be a
vertex. Let $%
p_{n}(o,A)$ denote the probability that a random walk of length $n$ starting
at $o$ ends in the finite vertex set $A$. Then with the spectral radius $\rho(G)$ we have
\[
p_{n}(o,A) \le \sqrt{|A|}\rho (G)^{n} +\frac{2|A|}{|G|}\text{.}
\]
\end{lemma}

\begin{proof} We prove the claim for finite graphs, the infinite case is similar but simpler.
Let $m=|G|,$ the number of vertices of $G$. Let $v_{0}$ denote the function on $G$ that takes
value $1/\sqrt{m}$ everywhere. Then $v_{0}M=v_{0}$. When $G$ is not bipartite, let $l_*^{2}(G)$ denote the
orthogonal subspace of $v_{0}$ in $l^{2}(G)$.  When $G$ is bipartite, let $\mathcal I$
be an independent subset of $G$ of size $m/2$ containing $o$ and let $v_{1}$
be the function on $G$ that takes values $1/\sqrt{m}$ on $\mathcal I$ and $-1/\sqrt{m}$ otherwise.
Then $v_{1}M^{n}=(-1)^nv_{1}$. Let $l_{\ast }^{2}(G)$ denote the subspace orthogonal to $v_0$ and $v_1$ in $l^{2}(G)$.

Now  $\rho (G)$ equals the norm of $M$ on $
l_{*}^{2}(G)$. Let $\delta_A$ denote the indicator function of the vertex set $A$. Let $v$ be a projection of $\delta_o$ onto $l_*^2(G)$, and let $v_*=\delta_o-v$.
Then $\|v\|\le 1$. For $G$ bipartite, we can write $v_*=a(v_0+v_1)$, with $a=1/\sqrt{m}$. We have
\begin{equation}\label{e:above}
\langle v_* M^n , \delta_A\rangle =
\langle  a(v_0+v_1)M^n, \delta_A\rangle =
\langle  a(v_0+(-1)^n)v_1, \delta_A\rangle
\end{equation}
Since $v_0$ and $v_1$ are orthonormal, writing $\delta_A$
in the orthonormal basis we see that \eqref{e:above} is bounded above by $$
a\langle v_0,\delta_A\rangle+a|\langle v_1,\delta_A\rangle| \le {2|A|}/m.
$$
Similarly, in the non-bipartite case
$\langle  v_*M^n, \delta_A\rangle = |A|/m$.
We now have
\begin{eqnarray*}
p_{n}(o,A)=\langle \delta_o M^{n},\delta_A\rangle =
\langle v_* M^n ,\delta_A \rangle+\langle v M^{n},\delta_A\rangle
\le 2|A|/|G|+\|v\|\cdot \rho(G)^{n}\cdot \|\delta_A\|.
\end{eqnarray*}
Here
$\|\delta_A\|=\sqrt{|A|}$. The claim follows.
\end{proof}

\subsection{Visits of nullcycles}

\begin{proposition}[Visits of nullcycles] \ \label{p:visits}
For any infinite $d$-regular rooted connected graph $(G,o)$ with $\rho(G)<1$ the number of visits $V_A$ to a finite vertex set $A$ of a random
nullcycle of length $n$ starting at $o$ satisfies $$\ev V_A\le 2\cdot 10^4\frac{|A|}{(1-\rho(G))^2}.$$
This is at most $ 10^7|A|$ if $\rho(G)\le 19/20$. Note that $19/20>\rho(\T_d)$ for $d\ge 3$.

For any finite $d$-regular graph $G$ we also have
$$
\ev V_A \le 4\cdot 10^4 |A|\left(\frac{1}{(1-\rho(G))^2}+\frac{72 n^2}{|G|}\right).
$$
  This is at most $2\cdot 10^{7}|A|$ if $\rho(G)\le 19/20$ and $n^2\le |G|$.
\end{proposition}

\begin{proof}
Let $X_j$ be a random cycle in the $d$-regular tree $\T_d$ started at the root  $o$,
and let $\bar X_j$ be its projection to the graph $G$.
Then we have
$$
\ev V_A= \ev \sum_{j=0}^n \one (\bar X_j\in A) =  \sum_{j=0}^n \pr(\bar X_j\in A).
$$
Condition on $|X_j|$, the distance from the root, and then sum over all possible options to get
$$
\ev V_A= \sum_{j=0}^n  \sum_{k=0}^n P(|X_j|=k) P(\bar X_j\in A \;|\;|X_j|=k).
$$
Note that
given $|X_j|=k$, the distribution of $X_j$ is uniform on the $k$-sphere about $o$ in the tree.
Thus the distribution on $\bar X_j$ in the graph $G$ is that
of the $k$th step of a nonbacktracking random walk. So let
$p_k$ denote the
probability that the $k$th step of the nonbacktracking walk is in $A$.

Switching the order of summation we get
$$
\ev V_A=\sum_{k=0}^n p_k \sum_{j=0}^n P(| X_j|=k)\le 500p_0+\sum_{k=1}^n 2\cdot 10^4 k p_k
$$
where the last inequality is based on the fact that the
$j$-sum gives the expected number of visits to distance $k$ for the random cycle
in $\T_d$, and the result of Lemma \ref{l:visits Td}. Note that $p_0=\one(o\in A)$.
The above can be bounded by Green function techniques as follows. Define
$$
\mathcal C(z) =\sum_{k=0}^\infty p_k\,z^k,
$$
the generating function for the proportion of nonbacktracking paths that start from $o$ and end in $A$.
For any  $z\in (0,1]$ we have
$$
\sum_{k=0}^nkp_k \le z^{1-n}\sum_{k=1}^\infty kp_k z^{k-1} =z^{1-n} C'(z)
$$
The right
hand side is a power series with nonnegative coefficients,
so it always makes sense but may equal $+\infty$.
Rewriting our bound in terms of $\mathcal C$ we get
$$
\ev V_A \le 2\cdot 10^4 z^{1-n} \mathcal C'(z)+500\cdot \one(o\in A).
$$
Let
$\mathcal G(z)$ be the analogous generating function for
simple random walk. It was shown in \cite{bartholdi},
(see formula (2.3) in \cite{OrtnerWoess}) that for any
$d$-regular graph we have
$$
\mathcal C (z)= \frac{\one(o\in A)}{d}+\frac{(d-1)^2-z^2 }{d
   \left(d-1+z^2\right)}\;\mathcal G\left(\frac{d\, z}{d-1+z^2}\right).$$
    Now with $x=dz/(d-1+z^2)$ we compute
$$
\mathcal C'(z) = a_0 \mathcal G(x) + a_1 \mathcal G'(x).
$$
where
\begin{eqnarray*}
a_0&=&-\frac{2 (d-1) z}{\left(d+z^2-1\right)^2} \le 0,
\\a_1&=&\frac{d^3-d^2 \left(z^2+3\right)+d \left(z^2+3\right)+z^4-1}{\left(d+z^2-1\right)^3}\le 1,
\end{eqnarray*}
for our range of parameters $d\ge 2$ and $z\in(0,1]$. We now consider two cases.

 1. For $G$ infinite with $\rho=\rho(G)<1$, we use the case $z=1$, noting that the radius of convergence of $\mathcal G$ is $1/\rho>1$.
Since $\mathcal G$ and its derivative are nonnegative, we get the upper bound
$$
\frac 12 10^{-4} \ev V_A \le   |A|+\mathcal G'(1)\le |A|+\bar{\mathcal G}'(1), \qquad \bar{\mathcal G}(z) =\frac{|A|}{1-z\rho}.
$$
The last inequality uses the fact that the probability that simple random walk at time $k$ is in  $A$ is bounded above by $|A|\rho^k$,
 so we can replace $\mathcal G'(z)$ by  $\bar{\mathcal G}'(z)$. Finally, we have
\[|A|+\bar{\mathcal G}'(1)
=
|A|\frac{1-\rho +\rho^2}{(1-\rho)^2}
\le \frac{|A|}{(1-\rho)^2}. \]

 2. For $G$ finite, we use the case $z<1$.
Since $\mathcal G$ and its derivatives are nonnegative, we get the upper bound
$$
C'(z)\le   \mathcal G'(x)\le \bar{\mathcal G}'(x).
$$
For the last inequality, we use $\rho=\rho(G)$,
$$
\bar{\mathcal G}(x) =|A|\sum_{k=0}^\infty x^k (\rho^k+2/{|G|})= \frac{2}{|G|} \frac{|A|}{1-x}+ \frac{|A|}{1-x\rho}.
$$
and use Lemma \ref{l:kicsi} to bound the return probabilities. This gives
$$
|A|+\bar{\mathcal G}'(x)=
\frac{2|A|}{|G|} \frac{1}{(1-x)^2}+
|A|\frac{\rho + (1-\rho x)^2}{(1-\rho  x)^2}\le \frac{2}{|G|} \frac{|A|}{(1-x)^2}+\frac{|A|}{(1-\rho )^2}.
$$
We now have
$$
\frac{1}{1-x} = \frac{d-1+z^2}{ (d-1-z)(1-z)} \le \frac{d}{d-2} \frac1{1-z}\le \frac3{1-z}
$$
and set $z=1-1/(2n)$ to get
$$
\ev V_A \le 2 \cdot 10^4 z^{-n} (\mathcal C'(z)+|A|)\le 2\cdot  10^4 (1-1/(2n))^{-n}|A|\left(\frac{2\cdot 3^2 \cdot 2^2 n^2}{{|G|}}+\frac{1}{(1-\rho)^2}\right)
$$
since for $n\ge 1$ the $(1-1/(2n))^{-n}\le 2$, and the claim follows.
\end{proof}

\subsection{Cycles and nullcycles}\label{c and n}
We now turn to the connection between ordinary cycles and nullcycles. We recall the definition
of nontrivial cycles.
\begin{definition}\label{d:trivial} Call a cycle of length $k$ in a graph  a
{\bf nontrivial cycle} if either
\begin{itemize}
\item for some directed non-loop edge $e$,
the number of times the cycle passes
through $e$ differs from the number of
times it passes through the reversal of $e$
\item or $k=1$.
\end{itemize}
\end{definition}
This definition differs slightly from ``vanishing in homology'', but is precisely what we need
in our proof (briefly we use $\mathbb Z$-homology for $k\ge 2$, and $\mathbb Z_2$-homology for $k=1$).
Our goal there is to take a nullcycle and make it non-nullhomotopic. We do this by swapping the direction of nontrivial sub-cycles of length $k\ge 2$. For loops this does not work (they do not have direction), so we have to have a separate argument for $k=1$: we add or erase them.

Cycles not covered by this definition are called {\bf trivial}. For example,
nullcycles are trivial and simple cycles are nontrivial.

The following
theorem is another main ingredient in the proof of Theorem \ref{vegesbecsles}. Let $\mathcal N_{n}$ denote the set of nullcycles starting at $o$ in the rooted graph $(G,o)$.

\begin{theorem}[Cycles and nullcycles]
\label{t:null and ordinary}Let $(G,o)$ be a $d$-regular rooted graph, and let $n,k,\ell>0$.

For a nullcycle  $w\in \mathcal{N}_{nk}$ let  $\chi(w,a,k)=\chi_\ell(w,a,k)$ denote the
indicator function that the path segment $w_{a},\ldots ,w_{a+k}$ is a nontrivial
$k$-cycle and that the vertex $w_a$ is visited at most
$\ell$ times by $w$. Let
\begin{equation}\label{yairstar}
\chi_w=\sum_{j=0}^{n-1} \chi (w,jk,k).
\end{equation}
Then with $c_{1}=1/16$ and $c_{k}=(d-1)^{-k}/2$ (for $k\ge 2$) we have
\begin{equation*}
\left\vert W_{nk}(o,o)\right\vert \geq \frac{1}{14}\sum_{ w\in \mathcal{N}_{nk}}
\exp\left(c_{k}\,\chi_w/\ell\right),
\end{equation*}
where $W_{nk}(o,o)$ is the set of cycles of length $nk$ starting at $o$.
\end{theorem}
The quantity $\chi_\omega$ will be estimated in
terms of the parameters $\gamma_k(G)$. Heuristically, if it is large, it means
that there are many different places in $\omega$ where rewiring is possible. The point in limiting the number of visits by $\ell$ is a convenient way to make sure that if there are many possible rewiring times, then they happen also at many different locations.

\begin{proof}
Let us denote $W=W_{nk}(o,o)$, and $\mathcal{N=N}%
_{nk}$, the subset of nullcycles.  We first break $W$ into equivalence classes, called
{\bf rewiring classes}. A loop is called {\bf single} if
its vertex has no other loops. Otherwise, we call it a {\bf multiple loop}.

When $k=1$ we break up the sum on the right of \eqref{yairstar} into a sum
over single loops and a sum over multiple loops, counted as
$\chi_{1w}+\chi_{2w}=\chi_w$. We choose $k$ (and for $k=1$
we choose single or multiple loops), and consider rewiring classes
depending on our choice.

\begin{description}
\item{Case $k=1$, single loops.}
Given a path $w$, let $\bar w$ denote the path in which all
self-loops whose vertex is visited at most $\ell$ times
(not counting consecutive visits) have been erased. Let $w\equiv w'$ if
$\bar w=\bar w'$. (``Not counting consecutive visits'' means that visits
to $v$ that are at consecutive times count as a single
visit.)

\item{Case $k=1$, multiple loops.} Two paths are
equivalent if for all times $i$ the vertices satisfy
$w_i=w'_i$, and $w$ and $w'$ agree except at times when
they traverse multiple self-loops.

\item{Case $k\ge 2$.} The paths $w$ and $w'$ are
equivalent if for all $0\le j\le n-1$ the following holds
\begin{itemize}
\item If $w_{jk}\neq w_{jk+k}$ then the path
segment between these times of $w$ and $w'$ is equal.
\item If $w_{jk}=w_{jk+k}$ and the path segment between these
times of $w$ is trivial, then it equals the
corresponding path segment in  $w'$.
\item If $w_{jk}=w_{jk+k}$ and the path segment between these
times of $w$ is nontrivial then it either equals
the corresponding path segment in $w$ or is the
time-reversal of that. We call $jk$ a {\bf proper cycle
time} of $w$, and the corresponding path segment a {\bf proper cycle} of
$w$.
\end{itemize}
This is illustrated in the example depicted in figure \ref{fig:equivalent_cycles}. 
\end{description}

For $w\in W$ let $[w]$ denote the equivalence class of $w$, called rewiring class. Note that the rewiring defined here is more complex
than the one in Section \ref{bm}.
For $w\in \mathcal{N}$ let $p(w)$ denote the probability
that a uniform random element of $[w]$ is nullhomotopic.

Then we have
\begin{equation*}
\left\vert W\right\vert =\dsum\limits_{A\text{ is a rewiring class}%
}\left\vert A\right\vert \geq \dsum\limits_{\substack{
A\text{ is a rewiring}
\\ \text{class, }A\cap \mathcal{N\neq \varnothing }}}\left\vert A\right\vert
=\dsum\limits_{w\in \mathcal{N}}\frac{\left\vert
[w]\right\vert }{\left\vert [w]\cap \mathcal{N}\right\vert
}=\dsum\limits_{w\in \mathcal{N}}p(w)^{-1}.
\end{equation*}%
What remains is to show that for all $w\in \mathcal{N}$ we
have
\begin{equation*}
p(w)\leq 14\exp(-c_{k} \chi_w/\ell).
\end{equation*}
We will do this case by case.

\smallskip

\noindent {\bf Case $k=1$, single loops.} We call a vertex
with a single loop (and its loop) {\bf reclusive} for $w$, if $w$
visits it at most $\ell$ times (not counting consecutive visits).
Whether a vertex is reclusive or not depends only on $[w]$.

Let $\tau_i,\, i=1,\ldots,\kappa$ denote the times when $\bar w$ visits a
reclusive vertex, and let $\phi$ be the number of loops erased
from $w$ to get $\bar w$. Then an element of $[w]$ is
determined by $X_1,\ldots X_\kappa$, the number of loops
inserted into $\bar w$ at times $\tau_1,\ldots,\tau_\kappa$. A
uniform random element of $[w]$ corresponds to a uniform
random choice of the $X_i$ so that their sum is $\phi$. Let
$\tr w$ denote the function that assigns to every reclusive
loop of $[w]$ the number of times modulo $2$ that $w$
passes through it. Then
$$
p(w)\le \pr(\tr w=0),
$$
where the right hand side refers to a random element of $[w]$. This is exactly the probability that for each reclusive
vertex the sum of the $X_i$ corresponding to that vertex is
even. By Lemma \ref{l:combinatorial} this is at most
$
14\exp \left(-\min(m,
 \phi/\ell)/14\right)
$, where $m$ is the number of different reclusive vertices
visited. Note that $m\ge \chi_{1w}/\ell$ and $\phi\ge \chi_{1w}$, so
we get the bound $ 14\exp
\left(-\frac{\chi_{1w}}{14\ell}\right). $

\smallskip

\noindent {\bf Case $k=1$, multiple loops.} We call a vertex {\bf important}
if it has a loop traversed by $w$. Further, we call a
loop important if its vertex is important (even if not traversed by $w$).
Note that
the set of important loops (or vertices) only depends on
the equivalence class of $w$.

For a path, let $\tr$ denote the function that assigns to
each important loop the number of times modulo $2$ that it
is traversed. Consider a random element $w$ of $[w]$. For
each important vertex $v$ with $k_v$ loops, let $\bar X_v=(X_{v,1},\ldots,
X_{v,k_v})$ record the number of times $w$ visits its
loops. Note $\bar X_v$ are independent as $v$ varies, and
each have a multinomial distribution with probabilities
$1/k_v$ for each option; each traversal is assigned to one of the loops uniformly at random.

Conditioning on the assignment of all traversals except for the last one, we
see that the probability that $X_{v,1},\ldots, X_{v,k_v}$
are all even numbers is at most $1/k_v\le 1/2$. So if
$i$ is the number of important vertices visited at most $\ell$ times, then we have
$i\ge\chi_{2w}/\ell$ and
$$
p(w)\le \pr(\tr w = 0) \le 2^{-i}\le 2^{-\chi_{2w}/\ell}.
$$

\smallskip

\noindent {\bf Case $k\ge 2$.} For a path, let $\tr$ denote the antisymmetric edge
function that sums $1$ over all forward steps of a path and
$-1$ over all backward steps (here ignoring self-loops).
Note that the trace of a random element $w$ in $[w]$ can be
written as
\begin{equation} \label{e:tr}
 \tr w = \tr \bar w + \sum_{\text{proper cycles $c$ of
$w$}} X_c\; \tr c
\end{equation}
where the $X_c$ are independent random variables uniform on
$\{-1,1\}$, and $\bar w$ denotes $w$ with all its proper
cycles removed. We claim that
$$
p(w) \le \pr(\tr w =0) \le 2^{-|w|_o}
$$
where $|w|_o$ is the maximum size of a subset of linearly
independent proper cycles of $w$. Indeed, consider such a
set $\mathcal C$, and complete it to a basis for
antisymmetric edge functions. Fix all values of $X_c$ for
$c\notin \mathcal C$. Then for $c\in \mathcal C$, looking
at the a $c$-coordinate of the equation \eqref{e:tr}, we
see that it can hold only if $X_c$ equals some fixed value,
which has probability $1/2$ or $0$, independently over the
coordinates. The claim follows.

\begin{figure}[h!]
\begin{center}
\includegraphics[width=5cm]{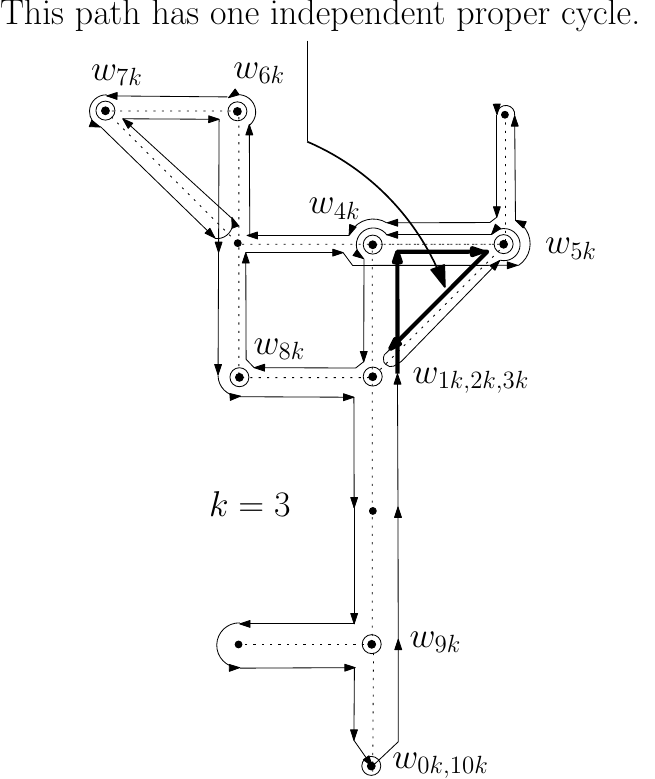}
\end{center}
\caption{A null cycle $\omega \in \mathcal{N}_{30}$ with two proper cycles of length $k=3$. These are opposite to each other and in particular dependent. Changing the direction in one of them gives raise to a nontrivial cycle equivalent to $\omega$. }
\label{fig:equivalent_cycles}
\end{figure}

Our next step is to bound the number of independent cycles.
Fix a $j_0$, and we consider the set $J$ of
indices $j$ so that the $\chi(w,jk,k)=
\chi(w,j_0k,k)=1$, and the cycles of $w$ at $jk$ and $j_0k$ share an edge. For a vertex $v$ let $J(v)$ denote the number of $j\in J$ so that $w_{jk}=v$.
Since for  $j\in J$ the vertex $w_{jk}$ is visited at most $\ell$ times, we have $J(v)\le \ell$.
If two $k$-cycles
share an edge, then a vertex on one and a vertex on the other are of distance at most $k-1$ from each other. Thus we have
$$
|J|=\sum_{v\in B(w_{j_0k},k-1)} J(v)
 \le \ell |B(w_{j_0k},k-1)|\le \ell d(d-1)^{k-2},
$$
where $B(v,r)$ is the ball of radius $r$ about $v$.
 This means that
the dependency graph of cycles has degree at most $d(d-1)^{k-2}\ell$ and size
$\chi_w$, and therefore contains an independent set of size
$\chi_w/(d(d-1)^{k-2}\ell+1)$. So we get $p(w)\le
2^{-\chi_w/(d(d-1)^{k-2}\ell+1)}\le e^{-\chi_w/(2(d-1)^k\ell)}$.

\bigskip

Now we have either $\chi_{1w}\ge  \frac78\chi_w$ or $\chi_{2w}\ge \frac18\chi_w$.
In either case, we get
$$
p(w)\le 14\exp(-\chi_w/(16 \ell)).
$$
Together with the $k\ge 2$ case, this completes the proof.
\end{proof}

The following simple probabilistic lemma was used in the proof of Theorem \ref{t:null and ordinary}.

\begin{lemma}\label{l:combinatorial}
Let  $X=(X_1, \ldots ,X_k)$ be a uniform random variable on
the set of $k$-tuples of nonnegative integers with even sum
$n\ge 2$.

(a) For any integer $k$-vector $x$ with $k\ge 2$ we have

\begin{eqnarray}
 \pr(X\equiv x \mod 2) &\le& \frac{\binom{n/2+k-1}{k-1}}
{\binom{n+k-1}{k-1}} \le
\exp\left(-\frac{1}{4/k+2/n}\right),
\end{eqnarray}
with equality at the first location if $x=0$.

(b) Consider a partition of $\{1\ldots k\}$ into $m$
nonempty parts
so that $k\le m\ell$ for some $\ell \ge 2$. Then with $\wedge$ denoting minimum, we have
$$\pr\left(
\sum_{i
 \in p} X_i \mbox{ is even for each part $p$}\right)\le 14\exp \left(-\frac{m\wedge
 (n/\ell)}{14}\right).
 $$
\end{lemma}

\begin{proof}

(a) (We thank P. Csikv\'ary for this simplification of our previous proof.)
To count the number of $k$ tuples that are equal to $x$ mod $2$, we subtract 1 from each odd entry and divide each resulting entry by 2. We get a bijection between such $k$-tuples and the number of $k$-tuples with entry sum $(n-o)/2$, where $o$ is the number of odd entries of $x$. Thus
\begin{eqnarray*}
 \pr(X\equiv x \mod 2) &=& \frac{\binom{(n-o)/2+k-1}{k-1}}
{\binom{n+k-1}{k-1}} \le \frac{\binom{n/2+k-1}{k-1}}
{\binom{n+k-1}{k-1}}.
\end{eqnarray*}
This shows the first inequality. For the second, note that
the right hand side equals
$$
\frac{n/2+1}{n+1}\;\frac{n/2+2}{n+2}\cdots
\frac{n/2+k-1}{n+k-1},
$$
Each factor is at most $
1-\frac{n/2}{n+k-1}$, giving a bound of
$$
\exp\left(-\frac{(k-1)n/2}{n+k-1}\right)\le
\exp\left(-\frac{1/2}{2/k+1/n}\right).
$$
The last inequality holds for $k\ge 2$.

 (b) Let $\bar X$ denote the
vector formed by the sums of the entries of $X$ over the
parts of our partition. Let $M\subset\{1,\ldots,k\}$ be a
subset of indices, one in each part, and let $M'$ be its
complement. Let $S=\sum_{i\in M} X_i$. Then
$$\ev S=\sum_{i\in M} \ev X_i=m\frac nk\ge \frac n\ell.$$ We first bound the probability
that $S$ is exceptionally small, namely that it is at most
$(k\wedge n)/(4\ell)$. $S$ has a discrete
beta distribution. By a standard construction, $S+m$ can be
realized as the time of the $m$th black sample when
sampling without replacement from $n$ white and $k-1$ black
balls. From this we get $$ p_s=P(S=s) =
\frac{\binom{s+m-1}{m-1}\binom{(n-s)+(k-m)-1}{(k-m)-1}}
{\binom{n+k-1}{k-1}}.
$$
We compute the ratio of these probabilities for two
consecutive values of $s$
\begin{equation}\label{e:probratio} \frac{p_s}{p_{s+1}}
=\frac{(s+1) (k-m+n-s-1)}{(m+s) (n-s)} \le \frac{s
(k-m+n-s)}{(m+s) (n-s)}+ \frac{k-m+n-s}{(m+s) (n-s)}.
\end{equation}
Assume that $s\le s_0=(m/2)\wedge (n/(2\ell))$. We first bound the
second term in \eqref{e:probratio}, which equals
$$
\frac{k-m}{m+s}\; \frac{1}{n-s} + \frac{1}{m+s} \le \ell
\frac{1}{n/2} + \frac{1}{m} \le \frac{3}{m\wedge (n/\ell)}
$$
since $n/2\le n-s$ and $k\le m\ell$. The first term in
\eqref{e:probratio} is increasing in $k$ so we substitute the smallest possible value
$k=m\ell$ to get the upper bound
$$
1-\frac{m (n-\ell s)}{(n-s) (m+ s)} \le 1- \frac{m
n/2}{n(m+m/2)} =\frac23.
$$
Thus when
\begin{equation}\label{e:omitted}
3 /(m\wedge (n/\ell))\le 1/12
\end{equation}
 the whole expression in \eqref{e:probratio} is
bounded above by $3/4$. Now note that from $s=s_0$ down the probability of $S=s$ decreases by at
least a factor of $3/4$. So
\begin{eqnarray*}\pr(S\le
(k\wedge n)/4\ell) &=&\sum_{s=0}^{s_0/2} p_s
\le \sum_{i=s_0/2}^\infty p_{(k\wedge
n)/2\ell}\left(\frac34\right)^{i} \\&\le& 4
\left(\frac{3}{4}\right)^{s_0/2}\le
4\exp \left(-\frac{m\wedge (n/\ell)}{14}\right).
\end{eqnarray*}
Condition on the random variables in $X_M'=(X_i,i\in M')$.
Given this information the random variable $X_M=(X_i,i\in
M)$ is uniform on the set of $k$-tuples of nonnegative
integers with sum $S$.
$$
\pr(\bar X = 0 \mod 2) = \ev [\pr( \bar X_M = \bar
X_{M'}\mod 2 \, |\, X_{M'})] \le \ev [\pr(\bar X_M = 0 \mod 2
\,|\, X_{M'})].
$$
The inequality follows from part (a). The last conditional
probability depends only on the value of $S$. Using part
$(a)$ we can break the expression up with $s=s_0/2$ as
\begin{align*}
\pr (S<s) &+ \ev \left[\ev\Big(\one(X_m=0 \mod 2) \one (S\ge
s)\,|\,S \Big)\right] \\&\le 4\exp \left(-\frac{m\wedge( n/\ell)}{14}\right) + \exp\left(-\frac{1}{4/k+2/s}\right)
\le 5\exp \left(-\frac{m\wedge( n/\ell)}{14}\right).
 \end{align*}
We increase the prefactor $5$ to $14$ in order to get a
trivial bound when \eqref{e:omitted} fails.
\end{proof}

\section{Explicit bounds on the spectral radius}
\label{s:explicit bounds on rho}

\subsection{A preliminary bound on the return probability}

\begin{proposition}\label{p:main for graphs}
Let $G$ be a $d$-regular unimodular random graph and let $k,\ell>0$. Then
with
$$c_k= \begin{cases}1/16&\mbox{for } k=1 \\ (d-1)^{-k}/2 & \mbox{for }k\ge 2 \end{cases}
$$ we have
 \begin{equation*}
 \ev \log \left\vert W_{nk}\right\vert
 \geq
 \log \left\vert \mathcal{N}_{nk}\right\vert-3 +\frac{c_kn}{\ell}\frac{1}{\left\vert \mathcal{N}_{nk}\right\vert }%
 \dsum\limits_{w\in \mathcal{N}_{nk}}
 \ev \chi_\ell (w,0,k).
 \end{equation*}%
\end{proposition}

\begin{proof} By Theorem \ref{t:null and ordinary} and the inequality of
arithmetic and geometric means we have
 \begin{eqnarray*}
\left\vert W_{nk}\right\vert &\geq& e^{-3}\sum_{ w\in \mathcal{N}_{nk}}
\exp\left(c_k\sum_{j=0}^{n-1} \chi (w,jk,k)/\ell\right)
 \\&\ge&
e^{-3}\left\vert \mathcal{N}_{nk}\right\vert \left(
\prod\limits_{w\in \mathcal{N}%
_{nk}}
\prod_{j=0}^{n-1} \exp\left(c_k \;\chi (w,jk,k)/\ell\right)
\right) ^{%
\frac{1}{\left\vert \mathcal{N}_{nk}\right\vert }}.
 \end{eqnarray*}
Taking logarithm of both sides gives us
\begin{equation*}
 \log \left\vert W_{nk}\right\vert -
 \log \left\vert \mathcal{N}_{nk}\right\vert
 \geq -3+\frac{c_k}{\ell\left\vert \mathcal{N}_{nk}\right\vert }%
\dsum\limits_{w\in \mathcal{N}_{nk}}\dsum\limits_{j=0}^{n-1}
\chi (w,jk,k).\end{equation*}%
Taking expected value of both sides over the random graph we get
\begin{equation*}
 \ev \log \left\vert W_{nk}\right\vert -
 \log \left\vert \mathcal{N}_{nk}\right\vert
 \geq -3+\frac{c_k}{\ell \left\vert \mathcal{N}_{nk}\right\vert }%
\dsum\limits_{w\in \mathcal{N}_{nk}}\dsum\limits_{j=0}^{n-1}
\ev \chi (w,jk,k).\end{equation*}%
We will use the Mass Transport Principle to show that the expression
\begin{equation} \label{av_Echi}
\dsum\limits_{w\in \mathcal{N}_{nk}}
\ev \chi (w,jk,k).
\end{equation}%
does not depend on the position $j$. Let the mass transport be defined as
\begin{equation*}
f(G,x,y)=\dsum\limits_{w\in \mathcal{N}_{nk}(x)}\mathbf{1}(w_{(n-j)k}=y)\chi (w,0,k)=\dsum\limits_{w\in \mathcal{N}_{nk}(y)}\mathbf{1}%
(w_{jk}=x)\chi (w,jk,k)
\end{equation*}
That is, for every nullhomotopic path $w$ starting at $x$, $x$ sends mass $%
\chi (w,0,k)$ to the $(n-j)k$-th position of $w$. The second
equality follows by rooting the path at $y$ instead of $x$. Trivially, the
mass transport does not depend on the root of $G,$ so the Mass Transport
Principle gives us
\begin{equation*}
\dsum\limits_{y\in V(G)}\mathbf{E}f(G,o,y)=\dsum\limits_{x\in V(G)}\mathbf{E}%
f(G,x,o)
\end{equation*}%
that is, the expected mass sent from the root equals the expected mass
received by the root. Plugging in the corresponding equations, we get
\begin{equation*}
\dsum\limits_{w\in \mathcal{N}_{nk}(o)}\mathbf{E}\chi (w,0,k)=\dsum\limits_{w\in \mathcal{N}_{nk}(o)}\mathbf{E}\chi(w,jk,k)
\end{equation*}%
and we get that the expression (\ref{av_Echi}) does not depend on $j$. This proves the theorem.
\end{proof}

\begin{lemma}\label{l:density to visits}
Let $(G,o)$ be a $d$-regular rooted graph with $\rho(G)\le 19/20$. Let $\gamma_k(G,o)$ be the number of nontrivial cycles of length $k$ starting at the root $o$.
Let $n$ satisfy $2k\le n \le \sqrt{|G|}$ and let $\ell=6\cdot 10^8 (4d-4)^{k}$. Then we have
$$
\frac{1}{\left\vert \mathcal{N}_{n}\right\vert }%
 \dsum\limits_{w\in \mathcal{N}_{n}}\chi_\ell(w,0,k) \ge \frac{\gamma_k(G,o) }{30(4d-4)^{k}} .
$$
\end{lemma}
\begin{proof}
We may assume $\gamma_k(G,o)\ge 1$, otherwise the claim is trivial.
In this Lemma $G$ is fixed, so the probabilistic language for nullcycles will not cause confusion.
So let $w$ be a uniform random element of $\mathcal N_n$.

The probability that a random cycle of length $n$ in $\T_d$ traverses a
specific path for its first $k$ steps can be bounded below easily by requiring the path to retrace its steps in the following $k$ times.
If $r_n$ is the return probability of simple random walk in $\T_d$, then
the total number of paths that do this is given by
$
r_{n-2k}d^{n-2k},
$
so the probability is at least
$$
\frac{r_{n-2k}d^{n-2k}}{r_nd^n} \ge \frac1{15} (d\rho(\T_d))^{-2k}= \frac1{15}(4d-4)^{-k}=:p,
$$
and the inequality uses both sides of Lemma \ref{l:returns} (but one side in the special case $n=2k$).
So if $G$ has $\gamma_k(G,o)$ cycles of length $k$ at $o$, then
the event $A$ that $w$ passes through one of them in the first $k$
steps satisfies $\pr A\ge p\gamma_k(G,o)$. Let $V_o$ be the number
of times the random nullcycle $w$ traverses $o$. By Proposition \ref{p:visits}
we have
$$
\ev V_o\le  2\cdot 10^{7}=c, \qquad \ev(V_o|A)\le \frac{\ev V_o}{\pr A}\le \frac{c}{p\gamma_k(G,o)}.
$$
By Markov's inequality with $\ell= 2c/p$
$$
\pr(V_o\ge\ell\,|\,
A)\le \frac{\ev(V_o|A)}{\ell} \le \frac{1}{2\gamma_k(G,o)}\le \frac12.
$$
This implies (using probabilistic notation for averaging over $\mathcal{N}_{n}$)
$$
\ev \chi(w,0,k) = \pr(A,V_o\le \ell)=\pr(A) - \pr(V_o> \ell|A)\pr(A) \ge \frac{\pr(A)}2 \ge \frac{p\gamma_k(G,o)}2
$$
as claimed.
\end{proof}

\subsection{Main bounds on spectral radius}

The following theorem implies Theorems \ref{vegesbecsles} and \ref{unimod}.

\begin{theorem}[Main results]\label{t:main}
Let $(G,o)$ be a  $d$-regular unimodular random graph and let $k\ge 1$. Let $\gamma_k(G,o)$ be the number of nontrivial cycles of length $k$ starting at $o$. Let
$$
\nu_k=2\cdot 10^{11} 2^{4k}(d-1)^{3k}k.
$$
For $G$ {\bf infinite} a.s.\ we have
\begin{eqnarray}
 \label{e:c main rho infinite}
\mathbf{E}\log \rho (G)&\geq &\log \rho (\T_{d})+\frac{1}{\nu_k}\ev \gamma_k(G,o).
\end{eqnarray}
For $G$ infinite and ergodic, we have $\rho(G) \ge \rho(\T_d) e^{\ev \gamma_k(G,o)/\nu_k}$.

\noindent Let $G$ be a {\bf finite} connected $d$-regular graph with $|G|\ge d^7$. Then for the root $o$ chosen uniformly at random we have
\begin{equation}\label{e:c main finite rho}
\frac{\rho(G)}{\rho(\T_d)} \ge 1+\frac1{\nu_k} \ev \gamma_k(G,o) -
\frac{\frac 32 \log \log_{d-1} |G| +6}{\log_{d-1} |G|}.
\end{equation}
In particular, for finite Ramanujan graphs with $|G|\ge d^7$ we have

\begin{equation}\label{e:c main Rama}
\ev \gamma_k(G,o) \le \nu_k
\frac{\frac 32 \log \log_{d-1} |G| +6}{\log_{d-1} |G|}.
\end{equation}
\end{theorem}

\begin{proof}
Let  $nk$ be  even and  $n\ge 1$. First assume that  $G$ {which may be finite or infinite} satisfies $\pr(|G|\ge (nk)^2)=1$.
We will use Lemma \ref{l:density to visits}, which requires $\rho(G)\le 19/20$. We first take care of the other case. For \eqref{e:c main rho infinite} and \eqref{e:c main finite rho} we need tho show for every such $G$ we have
$$
\log\rho(G)\ge \log \rho(\T_d) + \gamma_k(G)/\nu_k.
$$
Since $\gamma_k(G)\le d^k$, this inequality follows from
$$\rho(G)\ge 19/20,\qquad \rho(\T_d) \le \rho (\T_3)=2\sqrt{2}/3,\qquad \gamma_k(G)/\nu_k \le \tfrac{1}{2\cdot 10^{11}}.$$

Now assume $\rho(G)\le 19/20$. By Proposition \ref{p:main for graphs} and Lemma \ref{l:density to visits}
for $\ell=6\cdot 10^{8} (4d-4)^{k}$, $n\ge 1$ with $c_1=1/16x$ and  $c_k=(d-1)^{-k}/2$ for $k\ge 2$ we have
 \begin{equation}\label{e:limrho}
 \ev \log \left\vert W_{nk}\right\vert
 \geq
 \log \left\vert \mathcal{N}_{nk}\right\vert-3 +\frac{c_k}{\ell}n\frac{\ev \gamma_k(G,o) }{30(4d-4)^{k}}
 \end{equation}%
where
$$
\frac{c_k}{30(4d-4)^{k}\ell k} \ge \frac{1}{\nu_k}.
$$
For the first claim \eqref{e:c main finite rho},
we divide \eqref{e:limrho} by $nk$ and use the bounded convergence theorem. The second claim \eqref{e:c main rho infinite} follows from the fact that for $G$ ergodic $\rho(G)$ is constant.

The bound on $\mathcal N_{n}$ of Lemma \ref{l:returns} now shows that
 \begin{eqnarray}\label{e:c main returns}
  \ev \log p_{nk}(o,o)
 &\geq&
 nk\log \rho(\T_d) - \frac32 \log (nk) - 4 +\frac{nk}{\nu_k}{\ev \gamma_k(G,o) }.
\end{eqnarray}

For $G$ finite and $d\ge 3$ we have
\begin{equation}\label{e:prelimrho}
|G|\ge d^7 \qquad \Rightarrow\qquad \rho(G)\ge 1/(d-1)^{5/6},
\end{equation}
which follows from $\rho(G)^2+2/|G|\ge p_2(o,o)\ge 1/d$, a consequence of Lemma \ref{l:kicsi}.
Note that a lower bound on $|G|$ is needed for \eqref{e:prelimrho} since the complete graph with
loops has $|G|= d$  and $\rho(G)=0$.

Assume $|G|\ge d^7$, and $\log_{d-1}|G|\ge 10k$. Set $n = 2\lceil \frac{1}{2k} \log_{d-1} |G|\rceil $ so that
$$\rho(G)^{nk} \ge (d-1)^{-\frac{5}{6}nk} \ge (d-1)^{-\frac{5}{6}(\log_{d-1}|G|+2k)}\ge 1/|G|.
$$
(Here the power $5/6$ from \eqref{e:prelimrho} is used to offset the effect of $\lceil \cdot \rceil$, and thus yield a cleaner final bound).
By Lemmas \ref{l:returns} and \ref{l:kicsi}, the left hand side of \eqref{e:c main returns} is at most
$$
\log(\rho(G)^{nk}+2/|G|)=nk \log \rho(G)+\log \left(1+\frac2{|G|\rho(G)^{nk}}\right)\le nk\log \rho (G) +\log 3.
$$
We use this and divide \eqref{e:c main returns} by $nk$ and get the lower bound
\begin{equation}\label{e:mfrp}
\log \rho(G)\ge \log \rho(\T_d) +\frac1{\nu_k} \ev \gamma_k(G,o) -
\frac{\frac 32 \log \log_{d-1} |G| +4+\log 3}{\log_{d-1} |G|}.
\end{equation}
This proves \eqref{e:c main finite rho} for the case $\log_{d-1}|G|\ge 10k$.

The rest of the proof is standard and can be skipped. Its goal is to remove the restriction $\log_{d-1}|G|\ge 10k$. The same argument as above, using the trivial comparison with $\T_d$ for walks of length $2\lceil \frac12 \log_{d-1}|G|\rceil$ gives the (suboptimal) Alon-Boppana type bound
\begin{equation}\label{e:sab}
 \log \rho(G)\ge \log \rho(\T_d) +\frac{\frac 32 \log \log_{d-1} |G| +4+\log 3}{\log_{d-1} |G|} \end{equation}
as long as $\log_{d-1}|G|\ge 10$. For $d\ge 4$, \eqref{e:prelimrho} can be improved to
\begin{equation}
|G|\ge d^7 \qquad \Rightarrow\qquad \rho(G)\ge 1/(d-1)^{3/4},
\end{equation}
and this yields that \eqref{e:sab} holds as long as $\log_{d-1}|G|\ge 6$. So both for $d=3$ and $d\ge 4$ we get that \eqref{e:sab} holds as long as  $|G|\ge d^7$. Equation \eqref{e:sab} implies \eqref{e:c main finite rho} if
$$
\frac{\gamma_k(G,o)}{\nu_k} \le \frac {2-\log 3}{\log_{d-1} |G|}.
$$
With the trivial bound $\gamma_k(G,o)\le d(d-1)^{k-1}$, in the case $\log_{d-1}|G|\le 10k$ this is implied by
$$
\frac{d(d-1)^{k-1}}{\nu_k} \le \frac {2-\log 3}{10k},
$$
which holds trivially.
\end{proof}

\subsection{Bounds for graphs close to the Ramanujan threshold}

The following theorem implies Theorem \ref{essgirth}.

\begin{theorem}[Short cycles in Ramanujan graphs]
\label{t:ess girth r graphs}
Let $\alpha>0$, $d\ge 3$ and consider finite, connected $d$-regular graphs $G$ that are
close to Ramanujan in the sense that
$$
\rho(G)\le \rho(\T_d) + \frac{1}{(\log |G|)^{\alpha}}.
$$
Fix $\beta,\eps>0$ so that $\beta+\eps<\frac{\alpha \wedge 1}{6\log(d-1)+8 \log 2}$, (for example $\beta=\frac{\alpha \wedge 1}{16\log(d-1)}$). Then as $|G|\to\infty$, the proportion of vertices in $G$
whose $\beta \log\log |G|$-neighborhood is not a $d$-regular tree is  $o((\log|G|)^{-\eps})$.
\end{theorem}

\begin{proof}
Note that if the $k=\beta \log\log|G|$ neighborhood of a vertex $v$ is not a tree,
then $v$ is contained in a nontrivial cycle  of length $2k$, or its $k$-neighborhood contains
a vertex with a loop. We rule out these two cases separately.

By Theorem \ref{t:main}, \eqref{e:c main finite rho}, we have
$$
 \ev \gamma_k(G,o)
 \le c(d-1)^{3k}2^{4k}k^3
 \left(
 \frac{\log\log|G|}
 {\log|G|}+\frac{1}{(\log|G|)^\alpha}\right)
$$
so if $k=2\beta \log\log|G|$, then the dominant factor is
$$
(\log|G|)^{-\alpha\wedge 1+4\beta\log(4d-4)}
$$
and this is  $o(\log|G|^{-\eps'})$ for some $\eps'>\eps$ since
$$\beta+\eps<\frac{\alpha \wedge 1}{\log((d-1)^{3}2^{4})}.$$
The inequality also holds uniformly for all smaller $k$ (with a uniform constant in the $o(\cdot)$ term), and summing over all such
we get that the expected number of nontrivial cycles at $o$ of length at most $2k$
is $o(\log|G|^{-\eps}\log\log |G|)\to 0$. This rules out the first option.

For the second option, we use a simple mass transport argument (see the proof of Theorem \ref{t:null and ordinary} for the formal setup). Let each vertex with a loop send
mass $k$ to all elements in its $k$-neighborhood. Then
the expected amount of mass sent is at most
$d(d-1)^{k-1}\ev \gamma_{1}(G)$. The amount of mass
received is the number of vertices with loops in the $k$-neighborhood, lets call this $N$.
So we have
$$
\ev N\le c(d-1)^k \left(
 \frac{\log\log|G|}
 {\log|G|}+\frac{1}{(\log|G|)^\alpha}\right)
$$
By the same argument as before, this is $o(\log|G|^{-\eps})$
with the above choice of $\beta$.
\end{proof}

\subsection{Weakly Ramanujan sequences}

We are ready to prove that a $d$-regular weakly Ramanujan sequence of finite
graphs converges to the $d$-regular tree. \bigskip

\begin{proof}[Proof of Theorem \ref{weakram}] Let $(G_{n})$ be a weakly
Ramanujan sequence of finite $d$-regular graphs. Assume by contradiction,
that it does not have essentially large girth. Then, by passing to a
suitable subsequence, there exists $c>0$ and $L>0$ such that the cycle
densities $\gamma_L(G_n)>c$.

By passing to a subsequence, we can also assume that $(G_{n})$ is
Benjamini-Schramm convergent. Let $G$ be the limit of $(G_{n})$.

We claim that $G$ is infinite a.s. Assume this is not the case, then there
exists $R>0$ such that $G$ has size $R$ with probability $p>0$. This means,
that with probability at least $p$, the $R+1$-ball around the root has the
same size as the $R$-ball. So, for large enough $n$, the same holds for all $%
G_{n}$ with $p/2$. That is, at least $\left\vert G_{n}\right\vert p/2$
vertices lie in a connected component of size at most $R^{\prime }$, where $%
R^{\prime }$ is the size of the $R$-ball in the $d$-regular tree. This
implies that the number of connected components of $G_{n}$ is at least $%
\left\vert G_{n}\right\vert p/2R^{\prime }$, hence,
\begin{equation*}
\mu _{G_{n}}(1)\geq \frac{p}{2R^{\prime }}\text{.}
\end{equation*}%
This contradicts the assumption that $(G_{n})$ is weakly Ramanujan. So, our
claim holds.

We claim that $G$ is Ramanujan a.s. By the proof of Proposition \ref{ekvivalens}, $%
\mu _{G_{n}}$ weakly converges to the expected spectral measure $\mu_G $,
which yields $\mu_G ([-\rho (T_{d}),\rho (T_{d})])=1$, and  this implies that
$\mu_{G,o} ([-\rho (T_{d}),\rho (T_{d})])=1$ a.s. Since the spectral radius equals
the radius of the support of the spectral measure $\mu _{G,o}$ for any
rooted connected graph $G$ (see \cite[Lemma 2.1]{kesten1}), this implies that $\rho
(G)\leq \rho (T_{d})$ a.s.\ and our claim holds.

Now using Theorem \ref{unimod}, $G=T_{d}$ a.s., that is, $(G_{n})$ converges
to $T_{d}$ and so by Proposition \ref{ekvivalens}, it has essentially large
girth, a contradiction. Our theorem holds. \end{proof}

%
%
%

\section{Spectral radius and the fundamental group -- a sharp bound}
\label{s:rho and pi1}

\subsection{Relations in deterministic graphs}

In this section we analyze the spectral radius of a fixed rooted $d$-regular
infinite graph using random walks on its fundamental group.

For a graph $G$ rooted at $o\in V(G)$ and an arbitrary finite multiset $N$ of cycles in $G$
starting at $o$, we will also use  $N$ to denote the corresponding Markov operator on the
fundamental group $\pi_{1}(G,o)$ (which is a free product of copies of $\mathbb Z$ and the group of order $2$),
where the step
distribution is the uniform measure on $N$.  Let
$\left\Vert N\right\Vert $ denote the operator norm. The adjoint of the operator $N$ is the operator
corresponding to the multi-set $N^{-1}=\left\{ w^{-1}\mid w\in N\right\}$.
The multi-set $N$ may not be closed
to taking inverses, so the Markov operator need not be self-adjoint. We will use the property
\begin{equation*}
\left\Vert N\right\Vert =\sqrt{\left\Vert NN^{-1}\right\Vert }.
\end{equation*}%

Let $G$ be a graph, let vertices \thinspace $x,y\in V(G)$ and $k>0$ let $%
W=W_{k}(x,y)$ denote the set of walks of length $k$ in $G$ starting at $x$
and ending at $y$. Let $o\in V(G)$, let $u$ be a walk from $o$ to $x$ and
let $v$ be a walk from $y$ to $o$. When $W$ is non-empty, let
\begin{equation*}
N=\left\{ uwv\mid w\in W\right\} \subseteq \pi _{1}(G,o)
\end{equation*}%
and let
\begin{equation}\label{e:kappa}
\kappa _{k}(x,y)=\left\Vert N\right\Vert \text{.}
\end{equation}
Now $\kappa _{k}(x,y)$ does not depend on the choice of $o\,,u$ and $v$,
because the multi-set
\begin{equation*}
NN^{-1}=\left\{ uw^{\prime }w^{-1}u^{-1}\mid w,w^{\prime }\in W\right\}
\end{equation*}%
(defined with multiplicites),
so the corresponding Markov operator is the conjugate of the operator
belonging to $WW^{-1}$ by the fixed element $u$.

Note that the norm $\kappa_k$ satisfies
\begin{equation*}
\kappa _{k}(x,y)^2=\rho (\mathrm{Cay}(\pi _{1}(G,x),WW^{-1}))\in
\lbrack 0,1].
\end{equation*}%

Let $\mathcal{N}_{k}$ denote the set of nullcycles of length $k$
starting at $o$ (see Definition \ref{d:nullcycle}).
The following lemma relates $\left| \mathcal{N}_{k} \right|/ \left| W_k(o,o) \right|$,
the probability that a random cycle of length $k$ is a nullcycle to
the spectral radius $\kappa _{k}$. This relation can be established also
with respect to paths connecting two vertices.

\begin{lemma}
\label{basic}Let $G$ be a $d$-regular graph rooted at $o$ and let $k>0$. Let
$\ x$ be a vertex in $G$, and let $w$ be a path of length $|w|$ from $x$ to $%
o.$ Then
\begin{equation*}
\left\vert (W_{k}(o,x)w)\cap \mathcal{N}_{k+|w|}\right\vert \leq \left\vert
W_{k}(o,x)\right\vert \kappa _{k}(o,x)\leq (d\rho (T_{d}))^{k+|w|}\text{.}
\end{equation*}%
In particular, with $x=o$ and $w$ trivial we have%
\begin{equation*}
|\mathcal{N}_{k}|\leq \left\vert W_{k}(o,o)\right\vert \kappa _{k}(o,o)\leq
(d\rho (T_{d}))^{k}\text{.}
\end{equation*}
\end{lemma}

\begin{proof} We have

\begin{equation*}
\left\vert W_{k}(o,x)w\cap \mathcal{N}_{k+|w|}\right\vert =\left\vert
W_{k}(o,x)w\right\vert \frac{\left\vert W_{k}(o,x)w\cap \mathcal{N}%
_{k+|w|}\right\vert }{\left\vert W_{k}(o,x)w\right\vert }
\end{equation*}%
and the second factor on the right hand side equals the one step return
probability of the random walk on $\pi _{1}(G,o)$ with uniform step
distribution on $W_{k}(o,x)w$, hence it is at most the spectral radius of
the corresponding Markov operator. This proves the left inequality in the
lemma.

Now consider
\begin{eqnarray*}
\left\vert W_{k}(o,x)w\right\vert ^{n}&=&\left\vert (W_{k}(o,x)w)^{n}\cap
\mathcal{N}_{n(k+|w|)}\right\vert \frac{\left\vert
(W_{k}(o,x)w)^{n}\right\vert }{\left\vert (W_{k}(o,x)w)^{n}\cap \mathcal{N}%
_{n(k+|w|)}\right\vert }\\&\leq& \left\vert \mathcal{N}_{n(k+|w|)}\right\vert
\frac{\left\vert (W_{k}(o,x)w)^{n}\right\vert }{\left\vert
(W_{k}(o,x)w)^{n}\cap \mathcal{N}_{n(k+|w|)}\right\vert }
\end{eqnarray*}%
The second factor on the right hand side equals the inverse of the $n$-step
return probability of the same random walk as above. Taking $n$-th roots and
the limit as $n$ goes to infinity gives us the right side inequality of the
lemma. \end{proof} \bigskip

\begin{theorem}
\label{szep}Let $G$ be a $d$-regular graph rooted at $o$ and let $n,k>0$.
Then
\begin{equation*}
\left\vert W_{nk}(o,o)\right\vert \geq \dsum\limits_{w\in \mathcal{N}%
_{nk}}\dprod\limits_{j=0}^{n-1}\kappa _{k}(w_{jk},w_{(j+1)k})^{-1}\geq \frac{%
\left\vert \mathcal{N}_{k}\right\vert ^{n}}{(d\rho (T_{d}))^{nk}}\left\vert
W_{k}(o,o)\right\vert ^{n}.
\end{equation*}%
This implies
\begin{equation*}
d\rho (G)\geq \left( \dsum\limits_{w\in \mathcal{N}_{nk}}\dprod%
\limits_{j=0}^{n-1}\kappa _{k}(w_{jk},w_{(j+1)k})^{-1}\right) ^{1/nk}.
\end{equation*}%
Moreover, when we take the limit of the right hand side as $k\rightarrow
\infty $ (and $n$ changing arbitrarily) we get equality.
\end{theorem}

\begin{proof} Let us denote $W=W_{nk}(o,o)$ and $\mathcal{N=N}%
_{nk}$. We say that $w^{\prime }\in W$ is a \emph{rewiring} of $w\in W$ if $%
w_{jk}^{\prime }=w_{jk}$ for $0\leq j\leq n-1$.

As an example following the line of proof below, all possible rewirings of a nullcycle of length $35=5 \cdot 7$ are shown in figure \ref{fig:rewiring} below. It should be helpful to refer to that figure while reading the proof.
Rewiring is an equivalence
relation, and for $w\in W$ let $[w]$ denote the equivalence class of $w$.
For $w\in \mathcal{N}$ let $p(w)$ denote the probability that a uniform
random element of $[w]$ is nullhomotopic.

\begin{figure}[ht]
\begin{center}
\includegraphics[width=9cm]{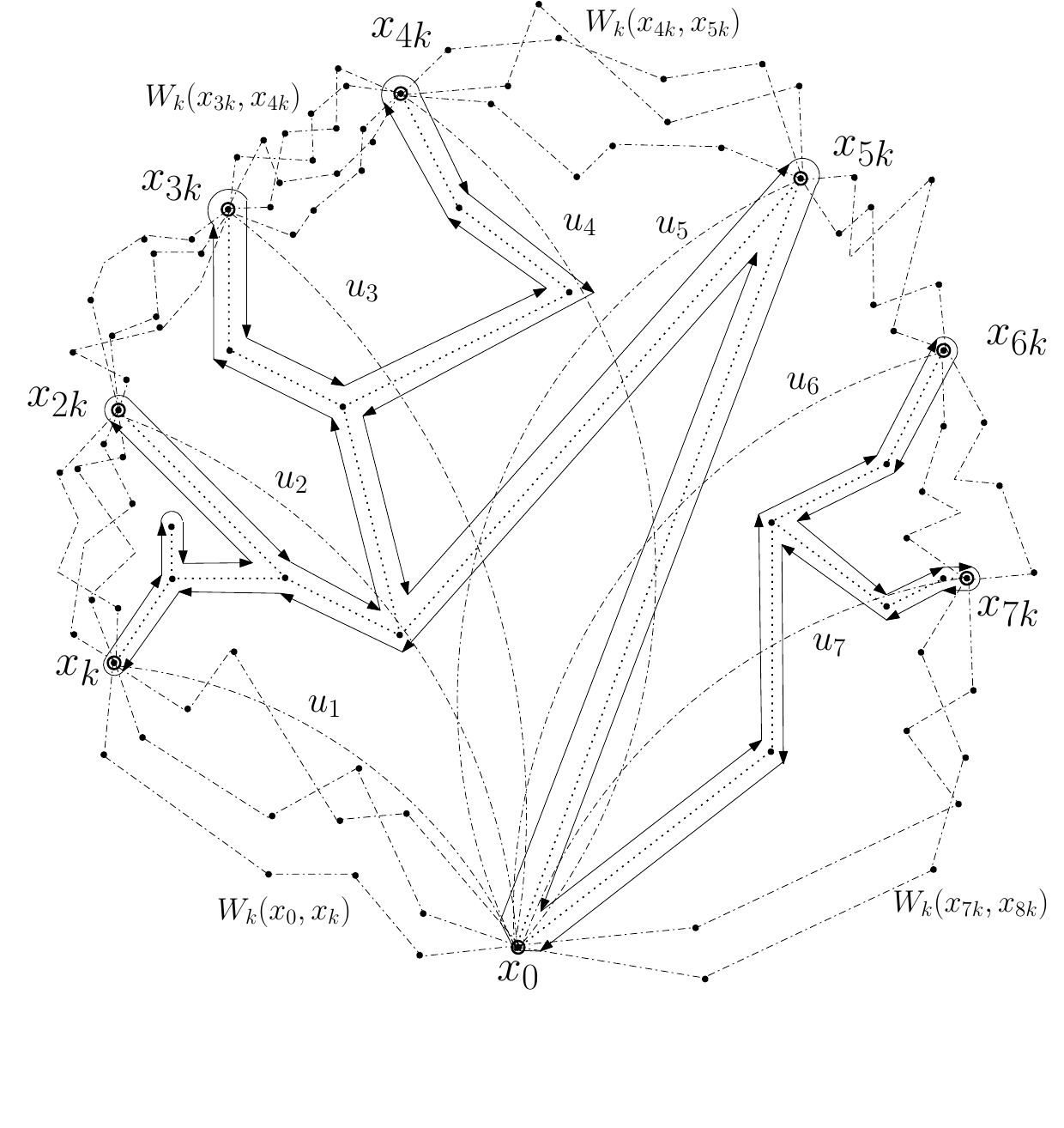}
\end{center}
\caption{Here we see how all  possible rewirings of a path $\omega \in \mathcal{N}_{35}$ are obtained upon replacing segments of length $k = 5$ by other possible replacements. The essential part of the proof is giving an estimate to the probability that such a rewiring $w'$ chosen at random will again be a nullcycle. Namely that the path $w'$ represents the trivial element of the fundamental group $\pi_1(G,x_0)$.}
\label{fig:rewiring}
\end{figure}

Then we have
\begin{equation*}
\left\vert W\right\vert =\dsum\limits_{A\text{ is a rewiring class}%
}\left\vert A\right\vert \geq \dsum\limits_{\substack{ A\text{ is a rewiring}
\\ \text{class, }A\cap \mathcal{N\neq \varnothing }}}\left\vert A\right\vert
=\dsum\limits_{w\in \mathcal{N}}\frac{\left\vert [w]\right\vert }{\left\vert
[w]\cap \mathcal{N}\right\vert }=\dsum\limits_{w\in \mathcal{N}}p(w)^{-1}.
\end{equation*}%
We claim that for all $w\in \mathcal{N}$ we have
\begin{equation*}
p(w)\leq \dprod\limits_{j=0}^{n-1}\kappa _{k}(w_{jk},w_{(j+1)k}).
\end{equation*}%
To prove this, for $0\leq j\leq n$ let $u_{j}$ be a path from $o$ to $w_{jk}$%
. Assume that $u_{0}$ and $u_{n}$ are the empty paths. For $0\leq j\leq n-1$
let
\begin{equation*}
N_{j}=\left\{ u_{j}wu_{j+1}^{-1}\mid w\in W_{k}(w_{jk},w_{(j+1)k})\right\}
\subseteq \pi _{1}(G,o)
\end{equation*}%
and let $v_{j}$ be a uniform random element of $N_{j}$. Let $N_{j}$ also denote
the Markov operator corresponding to the multi-set $N_{j}$. Then $\left\Vert
N_{j}\right\Vert =\kappa _{k}(w_{jk},w_{(j+1)k})$ by definition.

Now the random element $v=v_{0}\cdots v_{n-1}$ and the uniform random element
of $[w]$ have the same distribution as elements on the fundamental group. Indeed, they are related by adding or deleting
the nullcycles $u_{j}^{-1}u_{j}$. That
is, $p(w)$ equals the probability that $v$ is nullhomotopic. Let $e$ be the
characteristic vector of the identity element in $\pi _{1}(G,o)$. Using the
Cauchy-Schwarz inequality, this gives
\begin{eqnarray*}
p(w) &=&\left\langle e,e\dprod\limits_{j=0}^{n-1}N_{j}\right\rangle \leq
\left\langle
e\dprod\limits_{j=0}^{n-1}N_{j},e\dprod\limits_{j=0}^{n-1}N_{j}\right\rangle
^{1/2}\leq \\
&\leq &\left\Vert \dprod\limits_{j=0}^{n-1}N_{j}\right\Vert \leq
\dprod\limits_{j=0}^{n-1}\left\Vert N_{j}\right\Vert
=\dprod\limits_{j=0}^{n-1}\kappa _{k}(w_{jk},w_{(j+1)k})
\end{eqnarray*}%
and our claim holds.

Together with our first estimate on $\left\vert W\right\vert $ this completes
the proof of the first inequality of the theorem. For the second claim, note
that restricting the sum to nullcycles that return to $o$ at every
time $kj$ we get the lower bound

\begin{eqnarray*}
\dsum\limits_{w\in \mathcal{N}_{nk}}\dprod\limits_{j=0}^{n-1}\kappa
_{k}(w_{jk},w_{(j+1)k})^{-1} &\geq &\left\vert \mathcal{N}_{k}\right\vert ^{n}\kappa _{k}(o,o)^{-n}\\
 & \geq & \frac{%
\left\vert \mathcal{N}_{k}\right\vert ^{n}}{(d\rho (T_{d}))^{nk}}\left\vert
W_{k}(o,o)\right\vert ^{n}
\end{eqnarray*}%
\bigskip

Here the last inequality follows from Lemma \ref{basic}. \end{proof}

Let $G$ be a $d$-regular graph rooted at $o$. We define a new distribution
on the vertices of $G$ as follows. For $k,n>0$ where $n$ is even and $x\in
V(G)$ let $p(k,n,x)$ denote the probability that a uniform random
null-homotopic walk of length $n$ starting at $o$ is at $x$ at time $k$. Let
\begin{equation}\label{e:pdef}
p_{k}(x)=\lim_{n\rightarrow \infty }p(k,n,x)\text{.}
\end{equation}%
which, for each $k$ that describes
where the first $k$-segment of a the infinite bride of large length ends. The fact that this limit exists is a consequence of Corollary \ref{c:inifinite nullcycle}.

Let
\begin{equation}\label{e:kappa*}
\kappa _{k}^{\ast }(G,o)=\dprod\limits_{x\in V(G)}\kappa _{k}(o,x)^{p_{k}(x)}%
\text{.}
\end{equation}%
i.e. the geometric mean of the $\kappa _{k}(o,x)$ averaged over the vertices
$x$ with respect to the distribution $p_{k}$.

\begin{lemma}\label{l:sharp}
For any connected $d$-regular infinite graph $G$ we have
$$
 \rho(G)=\rho(\T_d) \lim_{k\to\infty} \kappa_{k}^{*}(G,o)^{-1/k}.
$$
Moreover, the terms $\kappa_k^*(G,o)^{-1/k}$ are bounded above by a constant depending on $d$ only.
\end{lemma}
\begin{proof}
For any vertex $x$ Lemma \ref{basic} gives the lower bound
\begin{eqnarray*}
\left\vert
W_{k}(o,x)\right\vert^{-1}\left\vert
\overline W_{k}(o,\overline x)\right\vert\le
\kappa _{k}(o,x) \leq \left\vert
W_{k}(o,x)\right\vert^{-1} (d\rho (T_{d}))^{k+|x|},
\end{eqnarray*}
where $\overline W$ is the function $W$ for the covering tree and
 $\overline x$ is a lift of $x$ corresponding to $w$ in that Lemma.
Using the simplest lower bounds for the number of paths we get
\begin{eqnarray*}
\rho(G)^{-k}\left\vert
\overline W_{k-|\overline x|}(o,o)\right\vert\le
\kappa _{k}(o,x) \leq
\left\vert W_{k-|x|}(o,o)\right\vert^{-1} (d\rho
(T_{d}))^{k+|x|}
\end{eqnarray*}
Note that $p(k,\cdot)$ assigns probability $q_k$ tending to 1 to vertices $x$ with $%
|x|\leq k^{2/3}.$
\begin{eqnarray*}
\rho(G)^{-k}\left\vert
\overline W_{k-k^{2/3}}(o,o)\right\vert^{q_k}\le
\kappa _{k}^*(G,o) \leq
\left\vert W_{k-k^{2/3}}(o,o)\right\vert^{-q_k} (\rho
(T_{d}))^{(k+k^{2/3})q_k}d^{k+k^{2/3}}
\end{eqnarray*}
The second claim follows by taking $k$th roots; the first follows by letting $k\to\infty$ and noting that the left and right hand sides both converge to $\rho
(T_{d})/\rho(G)$.
\end{proof}

\subsection{An asymptotically sharp bound}

\begin{theorem}\label{t:asymptotically sharp}
Let $G$ be a $d$-regular infinite unimodular random graph. Then for any  $k>0$ we have
\begin{equation*}
\mathbf{E}\log \rho (G)\geq \log \rho (T_{d})-\frac{1}{k}\mathbf{E} \log
\kappa^*_{k}(G,o)
\end{equation*}
and these bounds are sharp in the sense that
\begin{equation*}
\mathbf{E}\log \rho (G)= \log \rho (T_{d})-\lim_{k\to\infty}\frac{1}{k}\mathbf{E} \log
\kappa^*_{k}(G,o).
\end{equation*}
\end{theorem}

\begin{proof} By Theorem \ref{szep} and the inequality of
arithmetic and geometric means, we have
\begin{equation*}
\left\vert W_{nk}\right\vert \geq \dsum\limits_{w\in \mathcal{N}%
_{nk}}\dprod\limits_{j=0}^{n-1}\kappa _{k}(w_{jk},w_{(j+1)k})^{-1}\geq
\left\vert \mathcal{N}_{nk}\right\vert \left( \dprod\limits_{w\in \mathcal{N}%
_{nk}}\dprod\limits_{j=0}^{n-1}\kappa _{k}(w_{jk},w_{(j+1)k})^{-1}\right) ^{%
\frac{1}{\left\vert \mathcal{N}_{nk}\right\vert }}
\end{equation*}%
Taking logarithm of both sides gives us
\begin{equation*}
\log \left\vert W_{nk}\right\vert -\log \left\vert \mathcal{N}%
_{nk}\right\vert \geq \frac{-1}{\left\vert \mathcal{N}_{nk}\right\vert }%
\dsum\limits_{w\in \mathcal{N}_{nk}}\dsum\limits_{j=0}^{n-1}\log \kappa
_{k}(w_{jk},w_{(j+1)k})
\end{equation*}%
Taking expected value of both sides over the random graph we get
\begin{equation*}
\mathbf{E}\log \left\vert W_{nk}\right\vert -\log \left\vert \mathcal{N}%
_{nk}\right\vert \geq -\dsum\limits_{j=0}^{n-1}\frac{1}{\left\vert \mathcal{N%
}_{nk}\right\vert }\dsum\limits_{w\in \mathcal{N}_{nk}}\mathbf{E}\log \kappa
_{k}(w_{jk},w_{(j+1)k})
\end{equation*}%
We will use the Mass Transport Principle to show that the expression
\begin{equation}
\dsum\limits_{w\in \mathcal{N}_{nk}}\mathbf{E}\log \kappa
_{k}(w_{jk},w_{(j+1)k})  \label{izemize}
\end{equation}%
%
%
does not depend on the position $j$. Let the mass transport be defined as
\begin{equation*}
f(G,x,y)=\dsum\limits_{w\in \mathcal{N}_{nk}(x)}\mathbf{1}(w_{(n-j)k}=y)\log
\kappa _{k}(w_{0},w_{k})=\dsum\limits_{w\in \mathcal{N}_{nk}(y)}\mathbf{1}%
(w_{jk}=x)\log \kappa _{k}(w_{jk},w_{(j+1)k})
\end{equation*}
That is, for every nullhomotopic path $w$ starting at $x$, $x$ sends mass $%
\log \kappa _{k}(w_{0},w_{k})$ to the $(n-j)k$-th position of $w$. The second
equality follows by rooting the path at $y$ instead of $x$. Trivially, the
mass transport does not depend on the root of $G,$ so the Mass Transport
Principle gives us
\begin{equation*}
\dsum\limits_{y\in V(G)}\mathbf{E}f(G,o,y)=\dsum\limits_{x\in V(G)}\mathbf{E}%
f(G,x,o)
\end{equation*}%
that is, the expected mass sent from the root equals the expected mass
received by the root. Plugging in the corresponding equations, we get
\begin{equation*}
\dsum\limits_{w\in \mathcal{N}_{nk}(o)}\mathbf{E}\log \kappa
_{k}(w_{0},w_{k})=\dsum\limits_{w\in \mathcal{N}_{nk}(o)}\mathbf{E}\log
\kappa _{k}(w_{jk},w_{(j+1)k})
\end{equation*}%
and we get that the expression (\ref{izemize}) does not depend on $j$.

This gives
\begin{equation*}
\frac{\mathbf{E}\log \left\vert W_{nk}\right\vert -\log \left\vert \mathcal{N%
}_{nk}\right\vert }{nk}\geq \frac{-1}{k\left\vert \mathcal{N}_{nk}\right\vert
}\dsum\limits_{w\in \mathcal{N}_{nk}(o)}\mathbf{E}\log \kappa
_{k}(w_{0},w_{k})
\end{equation*}%
%
%
The right hand side now equals
\begin{equation*}
-\frac{1}{k} \mathbf{E}\dsum\limits_{x\in V(G)}p(k,nk,x)\log \kappa _{k}(o,x)
\end{equation*}%
with $p$ defined in \eqref{e:pdef}. For $G,k$ fixed, the right hand side is an average
of a bounded function $\log \kappa _{k}(o,x)$ on the vertices $x$ of $G$
with respect to the distribution $p(k,nk,\cdot )$. As $n\rightarrow \infty $, this distribution converges to the distribution $p_{k}(\cdot )$ by Corollary \ref{c:inifinite nullcycle},
and so does the corresponding average by the bounded convergence theorem.
Since each average is a bounded function of $G$, applying the bounded
convergence theorem again, now for the expectation over $G$, we get the
limiting inequality
\begin{equation*}
\mathbf{E}\log \rho (G)-\log \rho (T_{d})\geq -\frac{1}{k} \mathbf{E}\dsum\limits_{x\in
V(G)}p(k,x)\log \kappa _{k}(o,x)=- \frac{1}{k} \mathbf{E}\log \kappa _{k}^{\ast }(G,o).
\end{equation*}%
This completes the proof of the first claim of the theorem.
To prove the second claim, take expectation of the logarithm of the result of Lemma \ref{l:sharp}
and use the bounded convergence theorem. \end{proof}

\section{Graphs with uniformly dense short cycles\label%
{section_dense}}

In this section we prove Theorem \ref{distance}. This part of the paper is
independent of the rest as it does not use any of the results in the rest
and vice versa. Theorem \ref{distance} immediately implies that vertex
transitive Ramanujan graphs are trees; the proof for that \cite{paschke} is to
first show that every vertex transitive graph that is not a tree can be
covered by a Cayley graph that is also not a tree, and then use the original
Kesten's theorem. The proof presented here is purely combinatorial. It seems
tempting to try to prove Theorem \ref{t:main} using this method, but we did not
manage to do so. \bigskip

\begin{proof}[Proof of Theorem \ref{distance}.] Let $G$ be an infinite $%
d $-regular graph such that every vertex in $G$ has distance at most $R$
from a $k$-cycle. For a vertex $x\in G$ let $N(x)$ be the list of endpoints
of edges starting at $x$. For $n\geq 0$ let
\begin{equation*}
g(n)=\frac{d+(d-2)n}{d\sqrt{d-1}^{n}}
\end{equation*}%
Then $g(0)=1$ and for $n>0$ we have
\begin{equation*}
\frac{1}{d}\left( g(n-1)+(d-1)g(n+1)\right) =\frac{2\sqrt{d-1}}{d}g(n)
\end{equation*}%
Also, for $n\geq 0$ the function is monotonically decreasing, as

\begin{equation}
\frac{1}{\sqrt{d-1}}<\frac{g(n+1)}{g(n)}\leq \frac{2\sqrt{d-1}}{d}=\frac{g(1)%
}{g(0)}<1  \label{hanyados}
\end{equation}%
This is the spherical function that demonstrates $\rho (T_{d})\geq 2\sqrt{d-1%
}/d$.

Fix $o\in G$ forever. For $r\geq 0$ let
\begin{equation*}
S_{r}=\left\{ x\in G\mid d(o,x)=r\right\}
\end{equation*}
and for abbreviation let us denote $g_{r}=g(r)$.

For $x\in S_{r}$ let
\begin{equation*}
\deg ^{+}(x)=\left| N(x)\cap S_{r+1}\right| \text{, }\deg ^{0}(x)=\left|
N(x)\cap S_{r}\right| \text{ and }\deg ^{-}(x)=\left| N(x)\cap
S_{r-1}\right| .
\end{equation*}

Let the set of $\emph{return}$ $\emph{points}$ be defined as
\begin{equation*}
A=\left\{ x\in G\mid \deg ^{-}(x)+\deg ^{0}(x)\geq 2\right\}
\end{equation*}%
Let $k^{\prime
}=\left\lfloor R+k/2+1\right\rfloor $. By the assumption of the Theorem, the $%
k^{\prime }$-neighborhood of $A$ equals the whole $G$.

Let $R>0$ (this will tend to infinity later). Let us define $%
f_{R}:G\rightarrow \mathbb{R}$ as follows:
\begin{equation*}
f_{R}(x)=\left\{
\begin{array}{cc}
g(d(o,x)) & \text{if }d(o,x)\leq R \\
0 & \text{otherwise}%
\end{array}%
\right.
\end{equation*}%
Then $f_{R}\in l^{2}(G)$ and we have $\left\langle f_{R},f_{R}\right\rangle
=\sum_{r=0}^{R}\left\vert S_{r}\right\vert g_{r}^{2}$.

Let $x\in G$ and let $r=d(o,x)$.

If $r<R$ and $x\notin A$, then
\begin{equation*}
Mf_{R}(x)=\frac{1}{d}\left( g_{r-1}+(d-1)g_{r+1}\right) =\frac{2\sqrt{d-1}}{d%
}g_{r}
\end{equation*}%
otherwise
\begin{equation*}
Mf_{R}(x)=\frac{1}{d}\left( \deg ^{-}(x)g_{r-1}+\deg ^{0}(x)g_{r}+\deg
^{+}(x)g_{r+1}\right) \geq \frac{2\sqrt{d-1}}{d}g_{r}+\frac{1}{d}\left(
g_{r}-g_{r+1}\right) \text{.}
\end{equation*}

If $r=R$ then
\begin{equation*}
Mf_{R}(x)\geq \frac{1}{d}g_{R-1}\geq \frac{1}{d}g_{R}
\end{equation*}%
Using
\begin{equation*}
g_{r}-g_{r+1}\geq g_{r}(1-\frac{2\sqrt{d-1}}{d})=\frac{d-2\sqrt{d-1}}{d}g_{r}
\end{equation*}%
this gives us
\begin{eqnarray*}
\left\langle Mf_{R},f_{R}\right\rangle &\geq &\frac{2\sqrt{d-1}}{d}%
\sum_{r=0}^{R-1}\left\vert S_{r}\right\vert g_{r}^{2}+ \\
&&+\frac{d-2\sqrt{d-1}}{d^{2}}\sum_{r=0}^{R-1}\left\vert S_{r}\cap
A\right\vert g_{r}^{2}+\frac{1}{d}\left\vert S_{r}\right\vert g_{R}^{2}= \\
&=&\frac{2\sqrt{d-1}}{d}\sum_{r=0}^{R}\left\vert S_{r}\right\vert g_{r}^{2}+%
\frac{d-2\sqrt{d-1}}{d^{2}}\sum_{r=0}^{R-1}\left\vert S_{r}\cap A\right\vert
g_{r}^{2}- \\
&&-\frac{2\sqrt{d-1}-1}{d}\left\vert S_{R}\right\vert g_{R}^{2}
\end{eqnarray*}%
For each $x\in G$ let $a(x)\in A$ be a closest vertex in $A$. Then $%
d(x,a(x))\leq k^{\prime }$ and so evenly distributing the weight $%
g^{2}(d(o,a))$ on $a$ to all $x\in G$ with $a(x)=a$, we get
\begin{equation*}
\sum_{r=0}^{R-1}\left\vert S_{r}\cap A\right\vert g_{r}^{2}=\sum_{x\in A%
\text{, }d(o,x)\leq R-1}g^{2}(d(o,x))\geq \frac{1}{B}\sum_{x\in G\text{, }%
d(o,x)\leq R-(k^{\prime }+1)}g^{2}(d(o,a(x)))
\end{equation*}%
where $B=d((d-1)^{k^{\prime }}-1)/(d-2)$ is the size of the $k^{\prime }$%
-ball in $T_{d}$. On the other hand, \ref{hanyados}) implies
\begin{equation*}
\frac{g^{2}(d(o,a(x)))}{g^{2}(d(o,x))}>\frac{1}{(d-1)^{d(x,a(x))}}\geq \frac{%
1}{(d-1)^{k^{\prime }}}
\end{equation*}%
and so we get
\begin{equation*}
\sum_{r=0}^{R-1}\left\vert S_{r}\cap A\right\vert g_{r}^{2}>\frac{1}{%
B(d-1)^{k^{\prime }}}\sum_{r=0}^{R-(k^{\prime }+1)}\left\vert
S_{r}\right\vert g_{r}^{2}
\end{equation*}%
Putting together and trivially estimating $B$, we get
\begin{eqnarray*}
\frac{\left\langle Mf_{R},f_{R}\right\rangle }{\left\langle
f_{R},f_{R}\right\rangle } &>&\left( \frac{2\sqrt{d-1}}{d}+\frac{d-2}{%
d(d-1)^{2k^{\prime }}}\right) - \\
&&-\frac{C\sum_{r=R-k^{\prime }}^{R}\left\vert S_{r}\right\vert g_{r}^{2}}{%
\sum_{r=0}^{R}\left\vert S_{r}\right\vert g_{r}^{2}}
\end{eqnarray*}

where $C$ is an absolute constant. We get the required estimate if we show
that
\begin{equation*}
\lim_{R\rightarrow \infty }\frac{\left\vert S_{R}\right\vert g_{R}^{2}}{%
\sum_{r=0}^{R}\left\vert S_{r}\right\vert g_{r}^{2}}=0
\end{equation*}%
For $r\geq 0$ let $s_{r}=\left\vert S_{r}\right\vert /(d-1)^{r}$. Then
trivially $s_{r}\geq s_{r+1}$ and
\begin{equation*}
\left\vert S_{r}\right\vert g_{r}^{2}=\frac{1}{d^{2}}s_{r}(d+(d-2)r)^{2}
\end{equation*}%
thus we get
\begin{equation*}
\sum_{r=0}^{R}\left\vert S_{r}\right\vert g_{r}^{2}\geq \frac{1}{d^{2}}%
s_{R}\sum_{r=0}^{R}(d+(d-2)r)^{2}
\end{equation*}%
This gives us
\begin{equation*}
\frac{\sum_{r=0}^{R}\left\vert S_{r}\right\vert g_{r}^{2}}{\left\vert
S_{R}\right\vert g_{R}^{2}}\geq \frac{\sum_{r=0}^{R}(d+(d-2)r)^{2}}{%
(d+(d-2)R)^{2}}
\end{equation*}%
which tends to infinity with $R$. The theorem is proved. \end{proof}

\section{Examples of Ramanujan graphs}
\label{s:fin and infite rg}

\subsection{Tolerance of loops in Ramanujan graphs}

In this section we build examples of finite and infinite Ramanujan graphs
with some loops. It turns out that for infinite trees, there is a tolerance
phenomenon; the tree lets us insert some loops before giving up being
Ramanujan.

Recall that a Cayley graph of a group $G$ together with a finite set of generators $S=S^{-1}$
is the graph with vertex set $G$ and edge set $\{\{v,vs\}, s\in S\}$.
Our first result shows that every Cayley graph sequence that is Ramanujan
gives rise to another Ramanujan sequence with loops.

\begin{theorem}
\label{cayley}Let $G_{n}$ be an expander sequence of finite $d$-regular
Cayley graphs with $\left\vert G_{n}\right\vert \rightarrow \infty $. Then
there exists $H_{n}$ with $\left\vert H_{n}\right\vert \rightarrow \infty $
such that for all $n$, $H_{n}$ contains a loop and $G_{n}$ covers $H_{n}$.
In particular, $\rho (H_{n})\leq \rho (G_{n})$.
\end{theorem}

\begin{proof} Let $F$ be the free group with the alphabet $S$
and let $K_{n}$ be the normal subgroup in $F$ such that $G_{n}=\mathrm{Cay}%
(K_{n} \backslash F,S)$. Let $s\in S$ and let $F_{n}=\left\langle K_{n},s\right\rangle $
be the subgroup generated by $K_{n}$ and $s$. Let $H_{n}=\mathrm{Sch}%
(F_{n}\backslash F,S)$. Then the map between coset spaces $K_{n}g\mapsto F_{n}g$ is a
covering map from $G_{n}$ to $H_{n}$, since $F_{n}$ contains $K_{n}$. Every
eigenvector of $H_{n}$ can be pulled back to be an eigenvector of $G_{n}$,
which implies $\rho (H_{n})\leq \rho (G_{n})$. Also, $F_{n}s=F_{n}$, so $%
H_{n}$ contains a loop.

Assume now that when passing to a subsequence, $H_{n}$ has bounded size. Let
$N$ be the intersection of the $K_{n}$. Since $F$ has only finitely many
subgroups of a given index, $N \langle s \rangle$ has finite index in $F$.
Thus
$N\backslash F$ has a cyclic subgroup
of finite index,
hence it is amenable. Now a subsequence of the $G_{n}$ locally converges to
an infinite Cayley graph $G$ and $G$ is a quotient of $\mathrm{Cay}%
(N \backslash F,S)$%
, hence it is amenable as well. But then $G$ has a F\o lner
sequence, which then can be also found in the finite sequence. This implies
that $G_{n}$ is not an expander family, a contradiction. So $\left\vert
H_{n}\right\vert \rightarrow \infty $ as claimed. \end{proof}

Note that this proof only guarantees \emph{one} loop in $H_{n}$. The
known Lubotzky-Philips-Sarnak construction does not allow us to create two loops by factoring out
with two generators. For infinite graphs, the picture is very different.

\subsection{Infinite Ramanujan graphs are abundant }

Unlike finite Ramanujan graphs which are notoriously difficult to construct infinite Ramanujan graphs are abundant. In fact let $G$ be any graph whose degrees are bounded by $m$. There is a unique way of embedding $G$ into an $m$-regular graph $Y := \operatorname{Tree}_m(G)$ in such a way that the embedding $\iota: G \arrow Y$ induces an isomorphism on fundamental groups. In fact the graph $Y$ is constructed by ``gluing trees at every vertex'' in the unique possible way that would make the resulting graph $m$-regular.

Now fix a base vertex $o \in G \subset Y$ and let $W_n^Y(o,o)$ (resp, $V_n^Y(o,o)$) be the sets of $n$-cycles (resp, non-backtracking cycles) on the graph $Y$. The asymptotic of these are governed by the spectral radius $\rho(Y) = \frac{1}{m} \limsup_{n \arrow \infty} \left| W_n^Y(o,o)\right|^{1/n}$ and the co-growth $\alpha = \alpha(Y) = \limsup_{n \arrow \infty} \left|V_n^Y (o,o) \right|^{1/n}$. Now Grigorchuk's famous co-growth formula relates these two numbers by the following formula:
$$\rho = \left \{
   \begin{array}{lll}
      \frac{\sqrt{m-1}}{m} \left( \frac{\alpha}{\sqrt{m-1}} + \frac{\sqrt{m-1}}{\alpha} \right) & \qquad & {\text{if }} \alpha > \sqrt{m-1} \\
      \frac{2 \sqrt{m-1}}{m} & \qquad & {\text{otherwise }}
    \end{array} \right. .
$$
This formula is obtained by comparing the radii of convergence of the generating functions corresponding to these two types of random walks, see \cite[Equation 2.3]{OrtnerWoess}. This equation also plays a central role in our proof of Proposition \ref{p:visits}.

\begin{corollary} \label{inf_ram_criterion}\label{c:inf_ram_criterion}
Let $G$ be a graph with maximal degree bounded by $m$. Then $\operatorname{Tree}_m(G)$ is Ramanujan if and only if $m \ge \alpha^2(G)+1$. In particular if $G$ is $d$-regular then $\operatorname{Tree}_m(G)$ is Ramanujan whenever $m \ge d^2 -2d + 2$.
\end{corollary}
\begin{proof}
Clearly $\alpha(G) = \alpha(Y)$. The first statement follows, since by definition the graph $Y = \operatorname{Tree}_m(G)$ is Ramanujan if and only if it falls into the second clause of the above formula. The second statement follows since $\alpha(G) \le d-1$ for any $d$-regular graph.
\end{proof}

An open question of Itai Benjamini (private communication) asks
whether there exist infinite Ramanujan graphs where all
bounded harmonic functions are constant. This calls for different examples.

\section{A unimodular random graph of maximal growth}\label{s:growth}

For a rooted graph $G$ let $S_n$ denote the vertices at
distance $n$ from the root. Let
$$
\lgr G = \liminf_{n\to \infty} |S_n|^{1/n}.
$$
Clearly, for every $d$-regular graph $\lgr G\le d-1$, and
$\lgr \T_d=d-1$. The goal of this section is to prove
Theorem \ref{growth} from the introduction, namely to exhibit a $d$-regular unimodular
random graph $G$ different from $\T_d$ where $\lgr G=\lgr
\T_d=d-1$ almost surely.

For this, we consider site percolation on $\mathbb Z^2$,
namely a random induced subgraph where every vertex is
present with probability $p$ and absent with probability
$1-p$, independently. For $p$ large, the connected
component of the origin is infinite with positive
probability. Let $\overline \cluster$ denote the
distribution of the universal cover of the cluster given
that it is infinite; this is a tree with degree bounded by
$4$, but is not $4$-regular. It can be made $4$-regular
by adding loops.

\begin{theorem}\label{t:z2cluster}
The rooted random graph $\overline \cluster$ is a
unimodular random graph satisfying  $\lgr \overline \cluster =3$
with probability 1.
\end{theorem}
The following lemma follows from the
definition of unimodular random graphs.
\begin{lemma} The universal cover of a unimodular random graph
is a unimodular random graph.
\end{lemma}

Let $\cluster $ be a connected, induced subgraph of $\mathbb Z^2$,
and let $b_r$ be the size of the largest square fully contained in
$\cluster$ whose center is at distance at most $r$ in $\cluster $ from a fixed vertex.
Fix $a>0$, and consider the following property of $\cluster $
\begin{equation}\label{cluster property}
\liminf_{r\to\infty} \frac{b_r}{\log r} \ge a.
\end{equation}
It is clear that this property does not depend on the fixed
vertex.
Whether the infinite cluster in supercritical
percolation  has
this property is a tail event, so it has probability 0 or 1, although we will not use this.
We will argue for the latter.

\begin{lemma}
There is $a=a(p)$ so that the supercritical percolation cluster $\cluster$ satisfies
property \eqref{cluster property} with probability 1.
\end{lemma}

\begin{proof}
The fact that the set of open vertices in a percolation
cluster with $p>0$ satisfies
this property (with distance in $\ZZ^2$ instead of
distance in $\cluster $) is a simple exercise using independence
and the Borel-Cantelli lemmas.

We now use the two-round exposure technique, namely the following
construction of the set of open vertices of supercritical percolation at parameter $p$.
Take the union of open vertices in a
supercritical percolation with parameter $p'<p$, and an independent site percolation with parameter $p''$ where $p=p'+p''-p'p''$.

Consider the percolation at $p'$. Note that its infinite cluster
$\cluster '$ is unique and dense in $\mathbb Z^2$. Dense here means that the root (and so every vertex) has a positive probability of being contained in this cluster. Moreover, by the standard Antral-Pisztora result \cite{AP}, there is a constant $\eta$ so that the set of
vertices $\cluster ^+$ in $\cluster '$ whose distance  in $\cluster $ is at most $\eta$ times
their $\ZZ^2$ distance from the vertex in $\cluster $ closest to $0$
is also dense.

Given this dense set of vertices $\cluster ^+$, we can use the
independent percolation at $p''$ to add squares of size $c\log r$ at distance $r$
that are connected to $\cluster ^+$.
It follows that the infinite open cluster in the union of the two site percolations
has the desired properties.
\end{proof}

\begin{lemma} Let $\cluster$ be a connected subgraph of
$\mathbb Z^2$ satisfying property \eqref{cluster property}.
Then the probability that simple random walk exits from
$\cluster$ in $r$ steps decays slower than exponentially in
$r$.
\end{lemma}

\begin{proof}
Note that the probability that the random walk on $\ZZ^2$ starting at the center of a
square of volume $v$ in $\ZZ^2$, stays there for time at least
$t$ is bounded below by $q^{t/v}$ for some $q<1$.

So the probability that the random walk moves in $\cluster $ on a geodesic to
a square of size $c\log r$ at distance $r$, and there for time $r\log r$, is
at least $e^{-c'r}$. The claim follows.
\end{proof}

\begin{lemma}
Let $\cluster$ be a subgraph of a $d$-regular graph so that the
probability that the random walk stays in $\cluster$ for $n$ steps
decays slower than exponentially in $n$. Then the universal cover
of $\cluster $ has lower growth $d-1$.
\end{lemma}

\begin{proof}
Let $A_n$ denote the event that random walk stays in $\cluster$ for $n$ steps.
Let $s_n$ be the size of the sphere in the universal cover.
Then the probability of the event $B_n$ that nonbacktracking random walk on the base graph stays in  $\cluster $ until time $n$ is given by
$$
P(B_n)=\frac{s_n}{d(d-1)^{n-1}}.
$$
Note also that running ordinary random walk until time
$n$ and deleting the backtrackings, we get nonbacktracking random
walk run until a random time $N_n\le n$. Indeed, erasing the backtrackings
just means taking the geodesic from the starting point to
the current vertex in the universal cover tree.

Standard arguments
show that $N_n/n  \to 1-2/d$ and the event that
$N_n/n < \alpha$ for $\alpha<1-2/d$ fixed has probability that is exponentially small in $n$.
Thus we have
\begin{eqnarray*}
P(A_{n})&\le&\sum_{k=0}^n P(N_n=k)P(B_k) \le
P(N_n<\alpha n) + \sum_{k=an}^n P(N_n=k)P(B_{\alpha n})
\\&\le& P(N_n<\alpha n)+P(B_{\alpha n}),
\end{eqnarray*}
and therefore
$$
\frac{s_n}{d(d-1)^{n-1}}=P(B_n)\ge P(A_{n/\alpha})-P(N_{n/\alpha}<n),
$$
where the first probability decays slower than exponentially, and the second exponentially.
The claim follows.
\end{proof}

\begin{proof}[Proof of Theorem \ref{t:z2cluster}]
The component of the origin in the supercritical percolation in $\ZZ^2$ is unimodular, so it must be one even when conditioned to be infinite.
In this case, it satisfies property \eqref{cluster property}.
Then its universal cover is a unimodular random graph with
lower growth $d-1$.
\end{proof}

\noindent {\bf Acknowledgments.}
We thank an extremely careful referee for many useful comments on two
previous versions. M.A.\ is partially supported by MTA Renyi
"Lendulet" Groups and Graphs Research Group.
Y.G.\ was partially supported by ISF grant 441/11 and U.S. NSF grants DMS 1107452, 1107263, 1107367 ``RNMS: Geometric structures And Representation varieties" (the GEAR Network).
B.V.\ was supported by the NSERC Discovery Accelerator Grant and the Canada Research Chair program.

\end{document}